\documentclass[11pt]{amsart}

\usepackage{graphics}
\usepackage[dvips]{epsfig}
\usepackage{psfrag}

\usepackage[T1]{fontenc}
\usepackage[utf8]{inputenc}

\usepackage{amsmath}
\usepackage{amsfonts}
\usepackage{latexsym}
\usepackage{amssymb}
\usepackage[usenames]{color}
\usepackage{amsthm}
\usepackage{pinlabel}
\usepackage{textcomp}
\usepackage{subcaption}

\usepackage[all]{xy}

\theoremstyle{plain}
\newtheorem{defn}{Definition}[section]
\newtheorem{Thm}[defn]{Theorem}
\newtheorem{Prop}[defn]{Proposition}
\newtheorem{Lem}[defn]{Lemma}
\newtheorem{Cor}[defn]{Corollary}
\theoremstyle{remark}
\newtheorem{Rem}[defn]{Remark}

\newtheorem{Ex}[defn]{Example}

\numberwithin{equation}{section}

\newcommand{\R}{\ensuremath{\mathbb{R}}}
\newcommand{\Z}{\ensuremath{\mathbb{Z}}}

\begin{document}
\title[On torsion in (bi)linearized LCH in dim 3]{On torsion in (bi)linearized Legendrian contact homology in dimension 3}
\author[F. Bourgeois]{Fr\'ed\'eric Bourgeois} \address{Universit\'e Paris-Saclay, CNRS, Laboratoire de Math\'ematiques d'Orsay, 
91405 Orsay, France} \email{frederic.bourgeois@universite-paris-saclay.fr} 
\urladdr{https://www.imo.universite-paris-saclay.fr/{\raisebox{0.5ex}{\texttildelow}}frederic.bourgeois/}
\author[S. Connolly]{Salammbo Connolly} \address{Universit\'e Paris-Saclay, CNRS, Laboratoire de Math\'ematiques d'Orsay, 
91405 Orsay, France}
\email{salammbo.connolly@universite-paris-saclay.fr}
\urladdr{https://www.imo.universite-paris-saclay.fr/fr/perso/salammbo-connolly/}

\begin{abstract}
Linearized Legendrian contact homology (LCH) and bilinearized LCH are important homological invariants for Legendrian submanifolds in contact geometry. 
For legendrian knots in $\R^3$, very little was previously known about the possibility of having torsion in these invariants when they are defined over integer coefficients. In this paper, we give properties of torsion that can appear in linearized LCH with integer coefficients, and also give the full geography of bilinearized LCH with integer coefficients.
\end{abstract} 

\maketitle

\section{Introduction}
\label{sec:introduction}
Legendrian contact homology, introduced by Chekanov \cite{C}, is an invariant of Legendrian submanifolds of contact manifolds, up to Legendrian isotopy, and one of the most powerful tools in their study. It is the homology of a differential graded algebra $(\mathcal{A}(\Lambda),\partial)$ associated to a Legendrian $\Lambda$, called the Chekanov-Eliashberg DGA. However, due to the non-commutativity of $\mathcal{A}(\Lambda)$, it is in general very difficult to compute. It is therefore helpful to linearize the complex with the help of DGA augmentations. Given the choice of a unique augmentation $\varepsilon$, we get \textit{linearized Legendrian contact homology}, or $LCH^\varepsilon(\Lambda)$. If we choose two augmentations $\varepsilon_1$ and $\varepsilon_2$, we get \textit{bilinearized Legendrian contact homology}, or $LCH^{\varepsilon_1,\varepsilon_2}(\Lambda)$ (if $\varepsilon_1=\varepsilon_2$, we retrieve the linearized homology). In each case, while the homology does depend on the choices of augmentations, the collection of all homologies for all possible augmentations is an invariant of the Legendrian isotopy class.

To understand these invariants, it is useful to know their geography (ie. determining all possible values). For coefficients in $\Z/2\Z$, this problem was solved in the linearized case by Bourgeois, Sabloff, and Traynor \cite{BST}, and in the bi-linearized case by Bourgeois and Galant \cite{BG}. However, these constructions do not adapt to the case of integer coefficients because they do not create torsion. In fact, the possibility of there being torsion in (bi)linearized LCH is not well understood, and has been of interest in recent years. In higher dimensions, examples of torsion have been known for linearized LCH since \cite{EES05}. Golovko then showed in \cite{G} that, for dimensions $n=3$ or $n\geq 5$, any finitely generated group $G$ can appear as the degree $(n-1)$ linearized LCH group of a Legendrian in $\R^{2n+1}$ for a certain augmentation. This construction uses the torsion that appears in the singular homology of the Legendrian, and does not extend to Legendrians with simpler topology, such as knots. The question of whether or not any finitely generated abelian group can appear in the (bi)linearized LCH of a knot was therefore an open question, which was answered positively in the linearized case by Lipschitz and Ng in \cite{LN}, in degrees different than 0 and 1. Such results were still far from the full geography for the linearized case, and in the bilinearized case, no examples of torsion were known to us. 

In this paper, we look at Legendrian knots in $\mathbb{R}^3$, with the standard contact structure $\xi = \ker(dz-ydx)$. Our first result is a partial geography result for linearized Legendrian contact homology, which in particular extends the result of Lipschitz and Ng to the homology in degree 0, and gives more information about the homology in degree 1. 

Given a graded $\Z$-module $M$, we denote by $\overline{M}$ the graded $\Z$-module with the opposite grading, or
in other words $\overline{M}_k = M_{-k}$ for all $k \in \Z$. We also denote by $M[s]$ the graded $\Z$-module with its grading 
shifted by $s \in \Z$, or in other words $M[s]_k = M_{k+s}$ for all $k \in \Z$. In particular, the graded $\Z$-module freely generated
by a single element of degree $a$ is isomorphic to $\Z[-a]$.

\begin{Thm} \label{thm:linLCH}
For any Legendrian knot $\Lambda$ in $(\R^3, \xi_{\rm std})$ and any $\Z$-valued augmentation $\varepsilon$ of its Chekanov-Eliashberg 
DGA, their linearized Legendrian contact homology has the form
$$
LCH^{\varepsilon}(\Lambda) \simeq \Z[-1] \oplus F \oplus \overline{F} \oplus T \oplus \overline{T}[1],
$$
where $F$ is a finitely generated graded free $\Z$-module and $T$ is a finitely generated graded torsion $\Z$-module. \\
Conversely, for any finitely generated graded free $\Z$-module $F$ and any finitely generated graded torsion $\Z$-module $T$, there exist
an integer $d \ge 0$ (depending only on $T$), a Legendrian knot $\Lambda$ in $(\R^3, \xi_{\rm std})$ and a $\Z$-valued 
augmentation $\varepsilon$ of its Chekanov-Eliashberg DGA, such that
$$
LCH^{\varepsilon}(\Lambda) \simeq \Z[-1] \oplus F \oplus \overline{F} \oplus \Z^{2d}[0] \oplus T \oplus \overline{T}[1].
$$
\end{Thm}

We conjecture that we can in fact take $d = 0$ in the above result. This would determine the entire geography of linearized Legendrian 
contact homology with coefficients in $\Z$. Note that in order to prove this conjecture, it would suffice to find a single example of a
Legendrian knot $\Lambda_{\rm tor}$ equipped with a $\Z$-valued augmentation $\varepsilon_{\rm tor}$ such that 
$LCH^{\varepsilon_{\rm tor}}(\Lambda_{\rm tor}) \simeq \Z[-1]$, and allowing to create a pair of torsion generators in LCH by
interlacing two strands of $\Lambda_{\rm tor}$ with a Legendrian unknot.

In the case of bilinearized LCH, we obtain the entire geography of this invariant for Legendrian knots in dimension $3$.

\begin{Thm} \label{thm:bilinLCH}
For any Legendrian knot $\Lambda$ in $(\R^3, \xi_{\rm std})$ and any pair of non-dga-homotopic $\Z$-valued augmentations 
$\varepsilon_1, \varepsilon_2$ of its Chekanov-Eliashberg DGA, their bilinearized Legendrian contact homology has the form
$$
LCH^{\varepsilon_1, \varepsilon_2}(\Lambda) \simeq \Z[0] \oplus F \oplus T,
$$
where $F$ is a finitely generated graded free $\Z$-module of even rank and $T$ is a finitely generated graded torsion $\Z$-module. \\
Conversely, for any finitely generated graded free $\Z$-module $F$ of even rank and any finitely generated graded torsion $\Z$-module $T$, 
there exist a Legendrian knot $\Lambda$ in $(\R^3, \xi_{\rm std})$ and a pair of non-dga-homotopic $\Z$-valued augmentations $\varepsilon_1, 
\varepsilon_2$ of its Chekanov-Eliashberg DGA, such that
$$
LCH^{\varepsilon_1, \varepsilon_2}(\Lambda) \simeq \Z[0] \oplus F \oplus T.
$$
\end{Thm}

\subsection{Acknowledgments} This work was initiated as the Master thesis of the second named author under the supervision of the first named author.
We would like to thank Roman Golovko, as well as Robert Lipshitz and Lenny Ng, for
stimulating conversations about our respective works on this topic. 
Both authors were partially supported by the ANR project COSY (21-CE40-0002).

\section{Linearized and bilinearized Legendrian contact homology}
\label{sec:lbLCH}

\subsection{Chekanov-Eliashberg DGA}
\label{sec:CE-DGA}

Let $\Lambda$ be a Legendrian knot or link in $\R^3$ equipped with the standard contact structure $\xi_{std} = \ker(dz-y \, dx)$.
After a Legendrian isotopy, we can assume that the front projection of $\Lambda$ to the $xz$-plane has only transverse double points
and generic cusps. We fix a base point $*_i$ distinct from the double points and cusps, as well as an orientation, for each connected 
component $\Lambda_i$ 
of $\Lambda$. Throughout this paper, we will assume that the rotation number of each connected component of $\Lambda$ vanishes, 
so that its front projection admits a Maslov potential: each strand (i.e. connected component of the complement of the cusps) can be 
assigned an integer, so that when two strands meet at a cusp, the Maslov potential of the upper strand is one higher than the Maslov 
potential of the lower strand. Note that the Maslov potential of $\Lambda$ is uniquely defined up to the addition of a constant for each
connected component of $\Lambda$. Given two connected components $\Lambda_i$ and $\Lambda_j$ of $\Lambda$, we say that the
shift between their Maslov potentials is the difference between the Maslov potential at $*_j$ and the Maslov potential at $*_i$.

We now define Legendrian contact homology, mainly following the exposition from~\cite{N}.
Let $\mathcal{A}$ be the unital, noncommutative algebra freely generated over the ring of integers $\Z$ by the crossings and the right 
cusps in the front projection of $\Lambda$, as well as a pair of formal generators $t_i$ and $t_i^{-1}$ satisfying $t_i \cdot t_i^{-1} = t_i^{-1} \cdot t_i = 1$
for each connected component $\Lambda_i$ of $\Lambda$. Each generator is graded as follows: the formal variables $t_i$ and $t_i^{-1}$ have 
grading $0$, right cusps have grading $1$, and crossings are graded by the Maslov potential of the upper strand minus the Maslov potential 
of the lower strand, where the relative positions of the strands are considered on the left side of the crossing.

If $n \ge 0$, let $D^2_n = D^2 \setminus \{ x, y_1, \ldots, y_n \}$ be the closed disk with $n+1$ distinct punctures on its boundary, encountered 
in the given cyclic ordering when going around the boundary in the counterclockwise direction.
If $a, b_1, \ldots, b_n$ are crossings or right cusps in the front projection of $\Lambda$, let $\Delta(a;b_1, \ldots, b_n)$ be the collection of maps
$u: D^2_n \to \R^2_{xz}$ modulo reparametrization such that:
\begin{enumerate}
\item $u$ is an immersion on the interior of $D^2$,
\item  $u$ along the boundary of $D^2_n$ is an immersion, except at left cusps where the map $u$ is locally onto the right-facing limited by the cusp,
and except at right cusps where the map $u$ is locally two to one in left-facing region limited by the cusp and one to one in the rest of the neighborhood 
of the cusp,
\item $u$ sends $x$ to $a$ and maps a neighborhood of $x$ to the left-facing quadrant of $a$ if $a$ is a crossing, or to the interior of $a$
if $a$ is a right cusp,
\item $u$ sends $y_i$ to $b_i$ and maps a neighborhood of $y_i$ to the upward-facing or the downward-facing quadrant of $b_i$ if $b_i$ is a crossing,
or to the complement of the upper strand, or of the lower strand, or of the closure of the left-facing region limited by the cusp, if $b_i$ is a right cusp.
\end{enumerate}

For each $u \in \Delta(a;b_1, \ldots, b_n)$, we define $w(u) \in \mathcal{A}$ by
 $$
 w(u) = t(\eta_0) c(b_1) t(\eta_1) c(b_2) \ldots c(b_n) t(\eta_n),
 $$
where $c(b_i) = b_i^2$ if is a right cusp and the image of $u$ near $b_i$ is locally the complement of the closure of the left-facing region limited by the cusp
at $b_i$, and $c(b_i) = b_i$ otherwise. Moreover, $\eta_j$ is the image by $u$ of the arc in $\partial D^2_n$ limited by the boundary punctures $x_j$ and 
$x_{j+1}$ (with the convention that $y_0 = y_{n+1} = x$), so that $\eta_j$ is contained in the front projection of a single connected component $\Lambda_i$ 
of $\Lambda$. Then $t(\eta_j)$ is defined as $t_i^k$, where $k$ is the algebraic intersection number of $\eta_j$ with the base point $*_i$ in the oriented manifold
$\Lambda_i$. 

For each $u \in \Delta(a;b_1, \ldots, b_n)$, we also define $s(u)= (-1)^d$ where $d$ is the number of generators among $b_1, \ldots, b_n$ that are crossings
of even grading such that $u$ covers their downward-facing quadrant.

One can now define a differential $\partial: \mathcal{A} \to \mathcal{A}$ as a linear map lowering the grading by $1$ and satisfying the graded Leibniz rule
$\partial(cc') = \partial(c) c' + (-1)^{|c|} c \partial(c')$ for any $c, c' \in \mathcal{A}$ of pure grading, by the formula
$$
\partial(a) = r(a) + \sum_{\substack{n \ge 0, b_1, \ldots, b_n \\ u \in \Delta(a;b_1, \ldots, b_n)}} s(u)w(u),  
$$
where $r(a) = 1$ if $a$ is a right cusp and $r(a)=0$ if $a$ is a crossing.

The pair $(\mathcal{A}, \partial)$ is called the Chekanov-Eliashberg differential graded algebra (DGA) of $\Lambda$, it satisfies $\partial \circ \partial = 0$ and its
homology is invariant under Legendrian isotopy of $\Lambda$, provided the shifts between the Maslov potentials of its connected components remain unchanged.

\subsection{Linearized LCH}
\label{sec:lLCH}

The computation of the homology of the Chekanov-Eliashberg algebra is often very complicated due to the noncommutativity of $\mathcal{A}$, so that it is customary to use augmentations in order to
extract more computable invariants from this DGA. An augmentation of $\Lambda$ is a unital DGA map $\varepsilon: (\mathcal{A}, \partial) \to (\Z,0)$ preserving 
the grading. Let $C(\Lambda)$ be the $\Z$-module freely generated by the crossings and right cusps in the front projection of $\Lambda$, with the same grading 
as before. We define a linearized differential $\partial^\varepsilon: C(\Lambda) \to C(\Lambda)$ as a linear map lowering the grading by one, such that
\begin{eqnarray*}
\partial^\varepsilon (a) &=& \sum_{\substack{n \ge 0, b_1, \ldots, b_n \\ u \in \Delta(a;b_1, \ldots, b_n)}} \sum_{k=1}^n s(u)
\varepsilon((\eta_0) c(b_1) \ldots c(b_{k-1}) t(\eta_{k-1})) \ b_k \\ 
&& \hspace{3.5cm} \varepsilon(t(\eta_k) c(b_{k+1}) \ldots c(b_n) t(\eta_n)).
\end{eqnarray*}
The homology of this chain complex $(C(\Lambda), \partial^\varepsilon)$ is a graded $\Z$-module denoted by $LCH^\varepsilon(\Lambda)$ and called linearized
Legendrian contact homology (with respect to the augmentation $\varepsilon$). We say that two augmentations $\varepsilon$ and $\varepsilon'$ of $\Lambda$ 
are dga-homotopic if there exists an $(\varepsilon, \varepsilon')$-derivation $K: \mathcal{A} \to \Z$, i.e. $K(ab) = K(a) \varepsilon'(b) + (-1)^{|a|} \varepsilon(a) K(b)$ for
all $a, b \in \mathcal{A}$ of pure grading, raising the grading by $1$ and satisfying $\varepsilon - \varepsilon' = K \circ \partial$. It turns out that 
$LCH^\varepsilon(\Lambda)$ depends on $\varepsilon$ only through its dga-homotopy class. The collection of these graded $\Z$-modules for all (dga-homotopy 
classes of) augmentations of $\Lambda$ is invariant under Legendrian isotopy, again provided the shifts between the Maslov potentials of its connected 
components remain unchanged.

In \cite{L}, Leverson showed that for any link, an augmentation over a field must send $t$ to $-1$. It follows that the same is true for $\Z$-value augmentations. For the rest of this paper, for any choice of augmentation, we will then set $\varepsilon(t_i)=-1$.

\subsection{Bilinearized LCH}
\label{sec:bLCH}

Note that linearized LCH depends only on the abelianization of $(\mathcal{A}, \partial)$. In order to retain more information from the original DGA, one can use instead
a pair of augmentations $\varepsilon_1$ and $\varepsilon_2$, in order to define as in~\cite{BC} a bilinearized differential $\partial^{\varepsilon_1, \varepsilon_2}: C(\Lambda) \to C(\Lambda)$ as a linear map lowering the grading by one, such that
\begin{eqnarray*}
\partial^{\varepsilon_1, \varepsilon_2} (a) &=& \sum_{\substack{n \ge 0, b_1, \ldots, b_n \\ u \in \Delta(a;b_1, \ldots, b_n)}} \sum_{k=1}^n s(u)
\varepsilon_1((\eta_0) c(b_1) \ldots c(b_{k-1}) t(\eta_{k-1})) \ b_k \\ 
&& \hspace{3.5cm} \varepsilon_2(t(\eta_k) c(b_{k+1}) \ldots c(b_n) t(\eta_n)).
\end{eqnarray*}
The homology of this chain complex $(C(\Lambda), \partial^{\varepsilon_1, \varepsilon_2})$ is a graded $\Z$-module denoted by 
$LCH^{\varepsilon_1, \varepsilon_2}(\Lambda)$ and called bilinearized Legendrian contact homology (with respect to the augmentations $\varepsilon_1$ and 
$\varepsilon_2$). As above, $LCH^{\varepsilon_1, \varepsilon_2}(\Lambda)$ depends on $\varepsilon_1$ and $\varepsilon_2$ only through their respective dga-homotopy classes. In particular, if $\varepsilon_1$ and $\varepsilon_2$ are dga-homotopic, it is isomorphic to the linearized LCH of $\Lambda$ with respect to
$\varepsilon_1$, or to $\varepsilon_2$. The collection of these graded $\Z$-modules for all pairs of (dga-homotopy classes of) augmentations of $\Lambda$ is 
invariant under Legendrian isotopy, again provided the shifts between the Maslov potentials of its connected components remain unchanged.

\subsection{Duality}
\label{sec:duality}

One of the main properties satisfied by (bi)linearized LCH is some type of duality. This duality property was first established in dimension $3$ by Sabloff~\cite{S}, 
and was later generalized in higher dimensions by Ekholm, Etnyre and Sabloff~\cite{EES} in the form of a long exact sequence. We prefer to use the latter formulation,
as it emphasises better the structural constraints on (bi)linearized LCH. The duality exact sequence for linearized LCH reads 
\begin{equation} \label{eq:dualityseq-lin}
\ldots \to LCH^{n-k-1}_{\varepsilon}(\Lambda) \to LCH^\varepsilon_k(\Lambda) \stackrel{\tau_k}{\longrightarrow} H_k(\Lambda) 
\stackrel{\sigma_{n-k}}{\longrightarrow} LCH^{n-k}_{\varepsilon}(\Lambda) \to \ldots
\end{equation}
where $H_k(\Lambda)$ denotes the singular homology of the Legendrian knot or link $\Lambda$ with integer coefficients and vanishes except when $k =0$ or $1$.
In the case of bilinearized LCH, the duality exact sequence becomes~\cite{BC}
\begin{equation} \label{eq:dualityseq-bilin}
\ldots \to LCH^{n-k-1}_{\varepsilon_2, \varepsilon_1}(\Lambda) \to LCH^{\varepsilon_1,\varepsilon_2}_k(\Lambda) 
\stackrel{\tau_k}{\longrightarrow} H_k(\Lambda) \stackrel{\sigma_{n-k}}{\longrightarrow} LCH^{n-k}_{\varepsilon_2, \varepsilon_1}
(\Lambda) \to \ldots
\end{equation}
Let us now summarize the behavior of the maps $\tau_0$ and $\tau_1$ for these invariants when $\Lambda$ is connected, so that
$H_0(\Lambda) = H_1(\Lambda) = \Z$.

\begin{Prop}  \label{prop:tau}
Let $\Lambda$ be a Legendrian knot. If $\varepsilon$ is an augmentation of $\Lambda$, then
$\tau_0: LCH_0^\varepsilon(\Lambda) \to H_0(\Lambda)$ vanishes and $\tau_1: LCH_1^\varepsilon(\Lambda) \to H_1(\Lambda)$
is surjective. If $\varepsilon_1$ and $\varepsilon_2$ are augmentations of $\Lambda$ that are not dga-homotopic, then the map
$\tau_1: LCH_1^{\varepsilon_1, \varepsilon_2}(\Lambda) \to H_1(\Lambda)$ vanishes and the map 
$\tau_0: LCH_0^{\varepsilon_1, \varepsilon_2}(\Lambda) \to H_0(\Lambda)$ does not vanish.
\end{Prop}

Note that in the case of bilinearized LCH with $\varepsilon_1$ and $\varepsilon_2$ not dga-homotopic to each other, the map $\tau_0$ is not
necessarily surjective. This will be illustrated by Example~\ref{ex:tau0-bilin}.

\begin{proof}
The map $\tau_0$ was shown in~\cite[Proposition~3.2]{BG} to be given at chain level by $\varepsilon_2 - \varepsilon_1$, with coefficients in $\Z_2$. Let us 
carefully check that these signs are correct when working with coefficients in $\Z$.
From the point of view of the front projection, these two terms count immersed disks with boundary on the $2$-copy of $\Lambda$, with a positive puncture of 
grading $0$ at a crossing $c$ between the lower copy $\Lambda_-$ of $\Lambda$ to the upper copy $\Lambda_+$ of $\Lambda$, a negative puncture 
of grading $0$ at a crossing of either $\Lambda_+$ or $\Lambda_-$ very close to $c$, and a negative puncture of grading $-1$ at a crossing 
located near corresponding left cusps of $\Lambda_+$ and $\Lambda_-$. Since the second negative puncture has odd degree, it does not contribute to the
sign of the disk. The first negative puncture corresponds to a downward-facing quadrant (hence contributing a negative sign) if it corresponds to a crossing 
of $\Lambda_+$, and in this case this crossing comes just after $c$ on the oriented boundary of the disk, which therefore contributes $-\varepsilon_1$.
On the other hand, the first negative puncture corresponds to an upward-facing quadrant (hence contributing a positive sign) if it corresponds to a crossing 
of $\Lambda_-$, and in this case this crossing comes just before $c$ on the oriented boundary of the disk, which therefore contributes $+\varepsilon_2$.

If $\varepsilon_1$ and $\varepsilon_2$ are dga-homotopic, the above shows that $\tau_0$ vanishes in homology, as desired. Moreover, the definition of
the map $\sigma_k$ in~\cite{EES} shows that it is based on the count of a collection of disks that coincides with those counted by the map $\tau_k$. Therefore,
the map $\sigma_0$ vanishes as well. By exactness of~\eqref{eq:dualityseq-lin}, it follows that $\tau_1$ is surjective as desired.

If $\varepsilon_1$ and $\varepsilon_2$ are not dga-homotopic, the end of the proof of~\cite[Proposition~3.2]{BG} applies verbatim to show that the map
$\tau_0$ does not vanish. The image of $\tau_0$ is therefore a nontrivial $\Z$-submodule of $H_0(\Lambda) = \Z$, hence a free $\Z$-module.
Since $\tau_0$ and $\sigma_0$ count the same objects~\cite{EES}, this implies in turn that the map $\sigma_0$ defined on the free $\Z$-module 
$H_1(\Lambda) = \Z$ has a trivial kernel. By exactness of~\eqref{eq:dualityseq-bilin}, it follows that $\tau_1$ vanishes as desired.
\end{proof}

Note that the behavior of the maps $\tau_k$ was described in \cite{St} in the case of Legendrian links, but this will not be needed here as our main
results apply to Legendrian knots only.

\begin{Cor}\label{cor:torsion_iso}Let $\Lambda$ be a Legendrian knot, and let $\varepsilon_1,\varepsilon_2$ be two $\Z$-valued augmentations of its DGA. We denote the torsion part of $LCH^{\varepsilon_1,\varepsilon_2}_k(\Lambda)$ as $T_k(\Lambda,\varepsilon_1,\varepsilon_2)$. Then if either $\varepsilon_1\sim\varepsilon_2$, or $\tau_0$ is surjective, then $\forall k\in\Z$, $T_k(\Lambda, \varepsilon_1,\varepsilon_2)\cong T_{-k-1}(\Lambda, \varepsilon_2,\varepsilon_1)$. 
\end{Cor}

\begin{proof} The fact that for any $k\neq-1,0,1$, $LCH^{\varepsilon_1,\varepsilon_2}_k(\Lambda)\cong LCH^{-k}_{\varepsilon_2,\varepsilon_1}(\Lambda)$ is a direct consequence of the long exact sequence \eqref{eq:dualityseq-bilin}, and the fact that $H_k(\Lambda)=0$ for $k\neq0,1$. Since for any $k$, by the Universal Coefficient Theorem, the torsion part of $LCH^{k}_{\varepsilon_1,\varepsilon_2}(\Lambda)$ is isomorphic to the torsion part of $LCH_{k-1}^{\varepsilon_1,\varepsilon_2}(\Lambda)$, this then gives us $T_k(\Lambda, \varepsilon_1,\varepsilon_2)\cong T_{-k-1}(\Lambda, \varepsilon_2,\varepsilon_1)$ for any $k\neq-1,0$ (the case where $k=1$ is paired with the $k=-2$ case).

Suppose now that $\Lambda$ is connected, and that $\varepsilon_1$ and $\varepsilon_2$ are dga-homotopic. We may then write the homology as $LCH_k^{\varepsilon_1}$, as it is isomorphic to the linearized homology. We want to show that $T_0(\Lambda,\varepsilon_1,\varepsilon_2)\cong T_{-1}(\Lambda,\varepsilon_1,\varepsilon_2)$ (in this case the orders of the augmentations do not matter). The long exact sequence reads
 \begin{align*}
\lefteqn{\ldots \rightarrow LCH_{1}^{\varepsilon_1}(\Lambda) \overset{\tau_1}\rightarrow H_1(\Lambda) = \Z \overset{\psi_0}\rightarrow LCH^{0}_{\varepsilon_1}(\Lambda)} \\
& \hspace{3cm} \overset{\phi_0}\rightarrow LCH_{0}^{\varepsilon_1}(\Lambda) \newline \overset{\tau_0}\rightarrow H_{0}(\Lambda) = \Z\rightarrow \ldots
\end{align*}
Proposition \ref{prop:tau} tells us that $\tau_0$ vanishes, whereas $\tau_1$ is surjective. Since $\tau_0=0$, $\phi_0$ is surjective, and since $\tau_1$ is surjective, $\psi_0=0$, and so $\phi_0$ is injective. And so $LCH_{0}^{\varepsilon_1}(\Lambda)\cong LCH^{0}_{\varepsilon_1}(\Lambda)$. In particular, in this case $T_0(\Lambda, \varepsilon_1,\varepsilon_2)$ is isomorphic to the torsion part of $LCH^{0}_{\varepsilon_1}(\Lambda)$, which is isomorphic to $T_{-1}(\Lambda,\varepsilon_1,\varepsilon_2)$. 

Let us now consider the case where $\varepsilon_1$ and $\varepsilon_2$ are not dga-homotopic. Then Proposition \ref{prop:tau} tells us that $\tau_1$ vanishes, which gives us $LCH^{\varepsilon_1,\varepsilon_2}_1(\Lambda)\cong LCH^{-1}_{\varepsilon_2,\varepsilon_1}(\Lambda)$. Furthermore, the long exact sequence reads 
 \begin{align*}
\lefteqn{\ldots \overset{\tau_1}\rightarrow H_1(\Lambda) = \Z \overset{\psi_0}\rightarrow LCH^{0}_{\varepsilon_2,\varepsilon_1}(\Lambda)  \overset{\phi_0}\rightarrow LCH_{0}^{\varepsilon_1,\varepsilon_2}(\Lambda)} \\
& \hspace{3cm} \overset{\tau_0}\rightarrow H_{0}(\Lambda) = \Z\overset{\psi_{-1}}\rightarrow LCH^{1}_{\varepsilon_2,\varepsilon_1}(\Lambda)  \overset{\phi_{-1}}\rightarrow LCH_{-1}^{\varepsilon_1,\varepsilon_2}(\Lambda)\rightarrow0.
\end{align*}
Then if $\tau_0$ is surjective, then $\psi_{-1}$ must vanish, and so $\psi_{-1}$ must be injective. It is also clearly surjective due to $H_{-1}(\Lambda)$ being trivial, and so we have an isomorphism. $T_{-1}(\Lambda,\varepsilon_1,\varepsilon_2)$ must then be isomorphic to the torsion part of $LCH^{1}_{\varepsilon_2,\varepsilon_1}(\Lambda) $, which is isomorphic to $T_{0}(\Lambda,\varepsilon_2,\varepsilon_1)$.
\end{proof}

\subsection{Connected sums}
\label{sec:connsum}

Assume now that the Legendrian submanifold $\Lambda$ consists of several connected components $\Lambda_1, \Lambda_2, \ldots, \Lambda_r$.
We now describe a connected sum operation on this Legendrian link affecting two of these connected components, say $\Lambda_1$ and $\Lambda_2$.
After a Legendrian isotopy, we can assume that there exists a rectangular region $R$ in the $xz$-plane intersecting the front projection of $\Lambda$ in
exactly two horizontal line segments, the upper one in $\Lambda_1$ and the lower one in $\Lambda_2$, with the same Maslov potential. The connected sum
operations consists in the following modification of $\Lambda_1$ and $\Lambda_2$ inside $R$: the two line segments are replaced with a pair of smooth curves
approximating line segments, joining respectively the left end of the upper segment to the right end of the lower segment, and the left end of the lower segment 
to the right end of the upper segment. This is illustrated by Figure~\ref{fig:connsum}. Note that since the Maslov potential of the two original line segments coincide,
we have an induced Maslov potential of the resulting Legendrian submanifold $\widetilde{\Lambda}$.
Moreover, in this operation, a single new crossing $c$ is created, at the intersection point of the two new smooth curves. As their Maslov 
potential coincide, this crossing $c$ has grading $0$. If, after the connected sum, there does not exist any new holomorphic disks such that the boundary goes along one of these new smooth curves which passes through $c$ along $\Lambda$, without a corner at $c$, then any augmentation $\varepsilon$ of $\Lambda$ induces an augmentation $\widetilde{\varepsilon}$ of
$\widetilde{\Lambda}$ by setting $\widetilde{\varepsilon}(c) = -1$. Indeed, in this case any disk which contributes to the differential of the connected sum corresponds either exactly to a disk which existed before the connected sum, or to a disk which existed before with one additional puncture at $c$. 
So the differential is now of the form
\begin{align*}
\partial(a) = r(a) + &\sum_{\substack{n \ge 0, b_1, \ldots, b_n \\ u \in \Delta(a;b_1, \ldots, b_n)}} s(u) t(\eta_0) c(b_1) t(\eta_1) c(b_2) \ldots c(b_n) t(\eta_n)\\ 
&+\sum_{\substack{n \ge 0, b_1, \ldots, b_n \\ u \in \Delta(a;b_1, \ldots, c, \ldots, b_n)}} s(u) t(\eta_0) c(b_1) t(\eta_1) c(b_2) \ldots c \ldots c(b_n) t(\eta_n),
\end{align*} and the sign of $u$, when $u$ has a corner at $c$, is the same as the sign before the connected sum if the corner in $c$ points downwards, and the opposite if the corner points upwards. The connected sum removes the base point $*_2$ on $\Lambda_2$, which we may assume is in the area of $\Lambda_2$ bounded by $R$. The fact that $\varepsilon\circ\partial=0$ then implies that, for $\widetilde{\varepsilon}(c) = -1$, we have  $\widetilde{\varepsilon}\circ\partial = 0$, with the negative sign compensating for the loss of the negative sign given by the evaluation at $*_2$. Note that this argument still works if there does exist disks of which the boundary goes along one of the curves which passes through $c$ without turning in $c$, if the disk also turns in any crossing $b$ such that $\varepsilon(b)=0$.  

\begin{figure}
\labellist
\small\hair 2pt
\pinlabel {$\Lambda_1$} [bl] at 10 55
\pinlabel {$\Lambda_2$} [bl] at 10 3
\pinlabel {$c$} [bl] at 257 42
\endlabellist
  \centerline{\includegraphics[width=7cm]{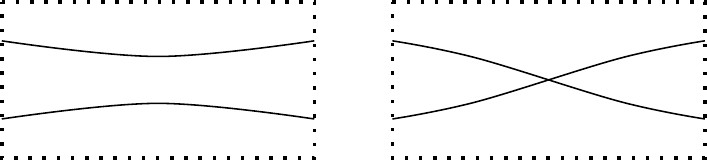}}
  \caption{Connected sum of two Legendrian knots.}
  \label{fig:connsum}
\end{figure}

The effect of this type of operation on bilinearized LCH was described in all dimensions with $\Z_2$-coefficients in~\cite[Proposition~3.5]{BG}. We now upgrade
this result to $\Z$-coefficients in the case of $1$-dimensional Legendrian submanifolds.

Consider the map $\tau_1: LCH^{\varepsilon_1, \varepsilon_2}_1(\Lambda) \to H_1(\Lambda)$ from the duality exact sequence~\eqref{eq:dualityseq-bilin}.
Since $H_1(\Lambda)$ is spanned by the fundamental classes $[\Lambda_i], i = 1, \ldots, r$ of the connected components of $\Lambda$, this map can
be written as $\tau_1 = \sum_{i=1}^r \tau_{1,i} [ \Lambda_i]$, using maps $\tau_{1,i}: LCH^{\varepsilon_1, \varepsilon_2}_1(\Lambda) \to \Z$.

\begin{Prop} \label{prop:connsum}
Let $\Lambda$ be a Legendrian submanifold in $\R^3$ equipped with two augmentations $\varepsilon_1$ and $\varepsilon_2$. Let $\widetilde{\Lambda}$ be the Legendrian submanifold obtained by performing a connected sum between two connected components $\Lambda_1$ and $\Lambda_2$ of $\Lambda$, and suppose that $\varepsilon_1$ and $\varepsilon_2$ extend to augmentations 
$\widetilde{\varepsilon}_1$ and $\widetilde{\varepsilon}_2$ for $\widetilde{\Lambda}$ by sending the new generator to $-1$.
Then if the map $\tau_{1,1} - \tau_{1,2}$ is surjective, then $LCH^{\widetilde{\varepsilon}_1, \widetilde{\varepsilon}_2}(\widetilde{\Lambda})$ and
$LCH^{\varepsilon_1, \varepsilon_2}(\Lambda)$ have the same torsion, but the latter has an additional free generator in degree $1$. If the map $\tau_{1,1} - \tau_{1,2}$ vanishes, then $LCH^{\widetilde{\varepsilon}_1, \widetilde{\varepsilon}_2}(\widetilde{\Lambda})$ and
$LCH^{\varepsilon_1, \varepsilon_2}(\Lambda)$ have the same torsion, but the former has an additional free generator in degree $0$.
\end{Prop}

\begin{proof}
Note that we can make the length of the Reeb chord corresponding to $c$ as small as we want, so that the graded 
$\Z$-module $\Z\langle c \rangle \simeq \Z[0]$ spanned by this crossing is a subcomplex of the (bi)linearized complex 
$C(\widetilde{\Lambda})$ for LCH. The quotient complex is naturally identified with the original complex
$C(\Lambda)$, so that we have a short exact sequence of complexes, inducing in homology the exact sequence
\begin{align*}
\lefteqn{0 \to LCH^{\widetilde{\varepsilon}_1, \widetilde{\varepsilon}_2}_1(\widetilde{\Lambda}) \to 
LCH^{\varepsilon_1, \varepsilon_2}_1(\Lambda) \stackrel{\rho_1}{\longrightarrow} \Z\langle c \rangle}  \\
& \hspace{3cm} \to LCH^{\widetilde{\varepsilon}_1, \widetilde{\varepsilon}_2}_0(\widetilde{\Lambda}) \to 
LCH^{\varepsilon_1, \varepsilon_2}_0(\Lambda) \to 0,
\end{align*}
as well as the isomorphisms $LCH^{\varepsilon_1, \varepsilon_2}_k(\widetilde{\Lambda}) \simeq LCH^{\varepsilon_1, \varepsilon_2}_k(\Lambda)$ for
$k \neq 0, 1$.
We claim that $\rho_1 = (\tau_{1,1} - \tau_{1,2})c$ and note that this Proposition directly follows from this claim.

Indeed, the map $\rho_1$ corresponds to the part of the bilinearized differential on $\widetilde{\Lambda}$ which goes to $c$. On the other hand, as explained in \cite[Section 3.2.5]{BC}, the map $\tau_1$ in the duality exact sequence for a link $\Lambda$ counts disks which have a left-facing corner at a generator $b$, and a marked point on the boundary which is mapped to a generic point $p_j$ of a connected component $\Lambda_j$ of $\Lambda$. These disks may also have downwards-facing corners, which we denote $b_1,…,b_k$ for those appearing before $p_j$ on the boundary of the disk, and $b_{k+1},…,b_r$ for those appearing after. If we call $\Delta(b;b_1,…,b_k,p_j,b_{k+1},…,b_r)$ the moduli space of such disks, then the map is given by
$$\tau_1(b)=\sum_j\sum_{\substack{r \ge 0, b_1, \ldots, b_r \\ u \in \Delta(b;b_1, …,b_k,p_j,b_{k+1},…,b_r)}} \varepsilon_1(b_1)…\varepsilon_1(b_k)\varepsilon_2(b_{k+1})…\varepsilon_2(b_r)[\Lambda_j].$$
Hence, due to the fact that $\varepsilon_i$ and $\widetilde{\varepsilon_i}$ match for all generators which are not $c$, the coefficient in front of $c$ in the bilinearized differential on $\widetilde{\Lambda}$ corresponds exactly to the difference between $\tau_{1,1}$ and $\tau_{1,2}$: disks which contribute to $\tau_{1,1}$ correspond to disks which cover the upward-facing quadrant of the crossing $c$, whereas disks which contribute to $\tau_{1,2}$ correspond to disks which cover the downward-facing quadrant of the crossing $c$, which adds a negative sign. This gives us the desired result $\rho_1 = (\tau_{1,1} - \tau_{1,2})c$. 
\end{proof}

\section{Torsion in linearized Legendrian contact homology} \label{sec:torsion_lin}

\subsection{Constraints on the linearized LCH module}
\label{sec:constraints_lin}

To understand the behavior of torsion in linearized contact homology, we must make sure that its properties are compatible with the duality long exact sequence \eqref{eq:dualityseq-lin}. Note that Corollary~\ref{cor:torsion_iso} implies that the torsion part of linearized Legendrian contact homology has the form $T \oplus \overline{T}[1]$
for some finitely generated graded torsion $\Z$-module $T$.

On the other hand, the constraints on the free part of linearized contact homology were described in~\cite[Theorem~1.1]{BST} for the case of 
generating family homology. The latter satisfies a duality exact sequence identical to~\eqref{eq:dualityseq-lin}, which is valid for any coefficient ring, and 
in particular for rational coefficients. The dimension of linearized contact homology with such coefficients gives the rank of its free part, as desired.
Let us denote by $r_k$ the rank of the free part of $LCH_{k}^{\varepsilon}(\Lambda)$. Then these constraints can be written as $r_k = r_{-k}$ for
$|k| > 1$ and $r_1 = 1 + r_{-1}$. This result was already established in~\cite{S} for $\Z_2$ coefficients, and the corresponding proof can be 
adapted to rational coefficients by adding to it the sign rules for the LCH differential. Note that since the Thurston-Bennequin invariant of $\Lambda$, 
which is the Euler characteristics of $LCH_{*}^{\varepsilon}(\Lambda)$, is always an odd integer for a $1$ dimensional Legendrian knot $\Lambda$, 
these constraints also imply that $r_0$ is even. These constraints imply that the free part of linearized contact homology has the form $\Z[-1] \oplus F
\oplus \overline{F}$ for some finitely generated graded free $\Z$-module $F$.

\subsection{Construction of torsion in linearized LCH} 
\label{sec:torsionlin}

Consider the Legendrian trefoil knot $\Lambda_{\rm tr}$ represented by Figure~\ref{fig:trefle}.
We have five generators, $a_1$, $a_2$, $b_1$, $b_2$, and $b_3$. The gradings are as follows:
\begin{gather*} |a_1| = |a_2| = 1,\\
\newline |b_1| = |b_2| = |b_3| = 0.\end{gather*}

\begin{figure}
\labellist
\small\hair 2pt
\pinlabel {$b_1$} [bl] at 275 410
\pinlabel {$b_2$} [bl] at 490 400
\pinlabel {$b_3$} [bl] at 730 380
\pinlabel {$a_1$} [bl] at 1050 540
\pinlabel {$a_2$} [bl] at 1050 260
\pinlabel{$\bullet$} [bl] at 560 60
\endlabellist
  \centerline{\includegraphics[width=7cm]{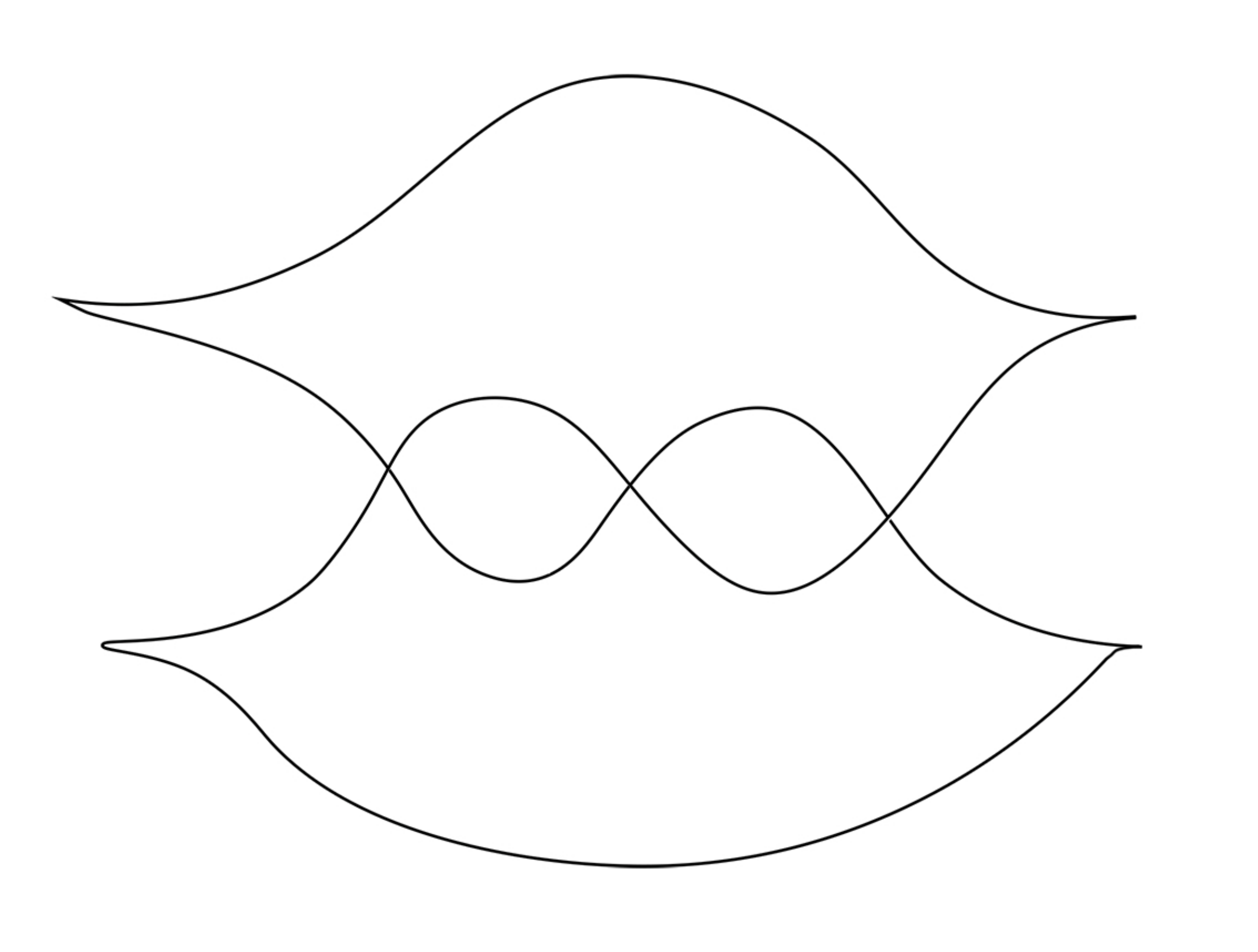}}
  \caption{Front projection of the trefoil knot.}
  \label{fig:trefle}
\end{figure}
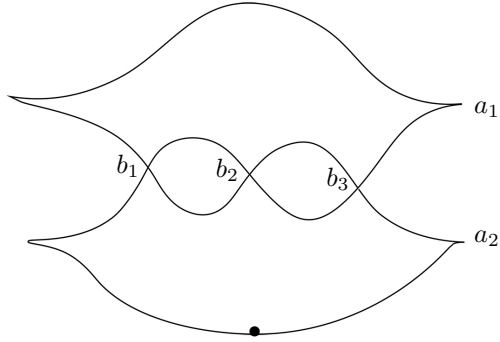

To choose an augmentation in $\Z$, we must compute the differential on the Chekanov-Eliashberg DGA $\mathcal{A}_{\Lambda_{\rm tr}}$. We have:
\begin{align*}\partial b_1 &= \partial b_2 = \partial b_3 = 0, \\
\partial a_1 &= 1 + b_1 + b_3 + b_1b_2b_3,\\
\partial a_2 &= 1 -b_1 t - b_3  t- b_3b_2b_1 t.\end{align*}

Since $a_1$ and $a_2$ are of non-zero grading, they must be augmented to 0. Also, since $\varepsilon\circ\partial$ must equal 0, we see that for any choice of augmentation for this DGA, $t$ must be augmented to $-1$, which is compatible with the result by Leverson \cite{L}. Such an augmentation $\varepsilon$ must therefore satisfy $$\varepsilon(b_1) + \varepsilon(b_3) + \varepsilon(b_1)\varepsilon(b_2)\varepsilon(b_3) = -1.$$   

If we take $\varepsilon(b_1)=-1$ and $\varepsilon(b_3)=0$, we can choose any value for $\varepsilon(b_2)$. 
For $n \in \mathbb{N}$, let us set consider the augmentation $\varepsilon_n$ defined by $\varepsilon(b_1)=-1$, 
$\varepsilon(b_2)=n\in\mathbb{N}$ and $\varepsilon(b_3)=0$.

With this choice of augmentation, the linearized differentials are:
\begin{align*}\partial^{\varepsilon_n} b_1 &= \partial^{\varepsilon_n} b_2 = \partial^{\varepsilon_n} b_3 = 0, \\
\partial^{\varepsilon_n} a_1 &= \partial^{\varepsilon_n} a_2= b_1 + b_3 -nb_3.\end{align*}

 We therefore have two homology generators in degrees 0, $[b_3]$ and $[b_2]$, and one in degree 1, $[a_1-a_2]$. 
 In particular, the corresponding linearized LCH is given by
 $$
 LCH^{\varepsilon_n}(\Lambda_{\rm tr}) = \Z[-1] \oplus \Z^2[0].
 $$

We now proceed to intertwine a Legendrian unknot $\Lambda_{1,k}$ with our trefoil knot, as seen in Figure \ref{fig:trefle_trivial}, such that there is a shift of $k \in\Z$ between the Maslov potential of $\Lambda_{\rm tr}$ and $\Lambda_{1,k}$. We extend $\varepsilon_n$ to 0 on all new generators. In consequence of this, the previously computed differentials do not change. 

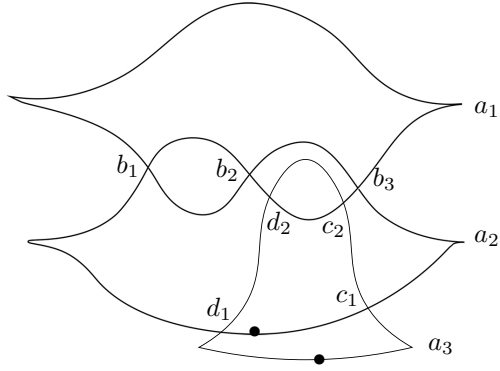
\begin{figure}[h]
\labellist
\small\hair 2pt
\pinlabel {$b_1$} [bl] at 275 410
\pinlabel {$b_2$} [bl] at 490 400
\pinlabel {$b_3$} [bl] at 830 380
\pinlabel {$a_1$} [bl] at 1050 540
\pinlabel {$a_2$} [bl] at 1050 260
\pinlabel{$\bullet$} [bl] at 560 60
\pinlabel {$a_3$} [bl] at 950 25
\pinlabel {$c_1$} [bl] at 750 130
\pinlabel {$c_2$} [bl] at 720 280
\pinlabel {$d_1$} [bl] at 470 100
\pinlabel {$d_2$} [bl] at 600 285
\pinlabel{$\bullet$} [bl] at 700 00
\endlabellist
  \centerline{\includegraphics[width=7cm]{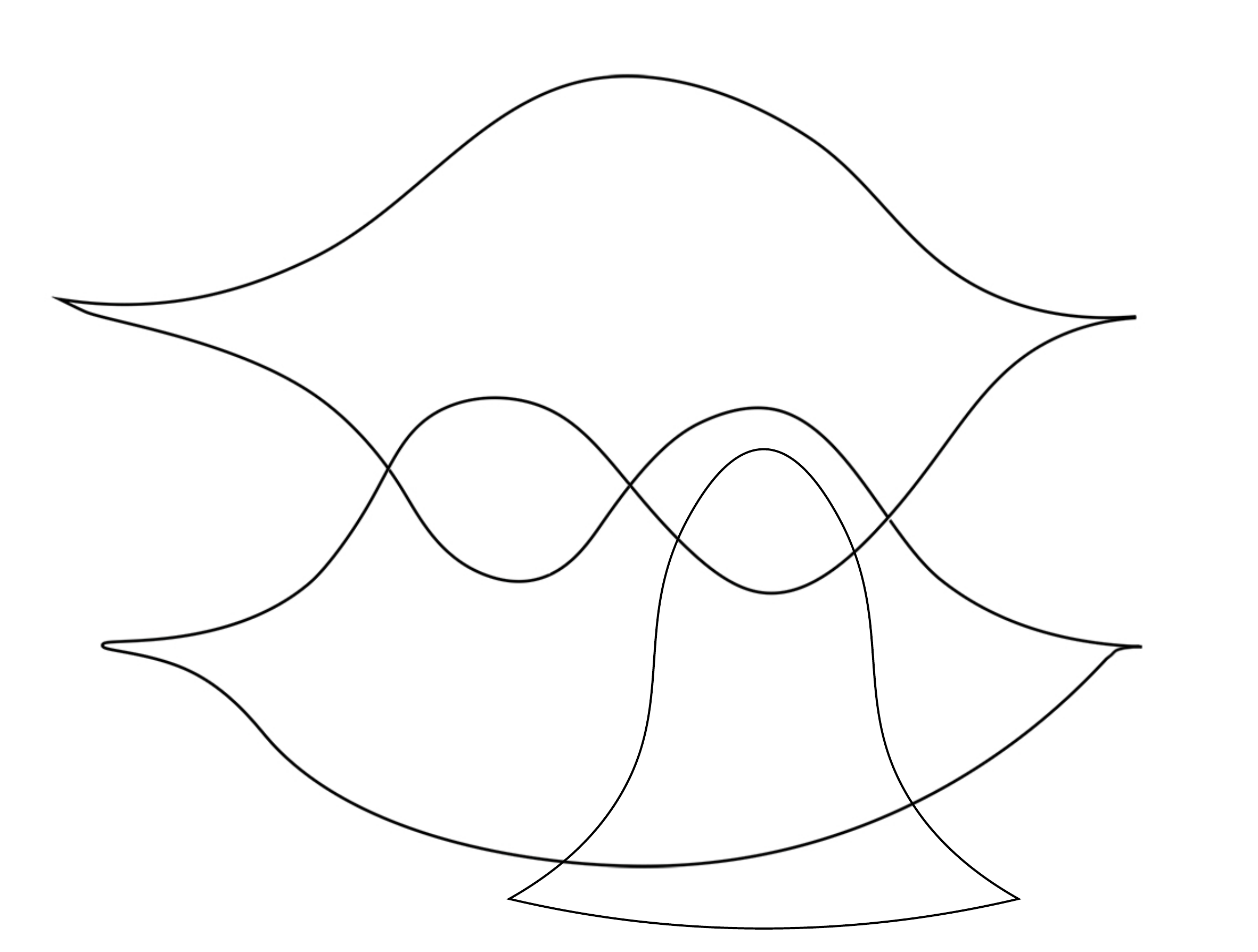}}
  \caption{Link between the trefoil knot and a legendrian unknot.}
  \label{fig:trefle_trivial}
\end{figure}

The gradings of the new generators are the following:
\[ |a_3| = 1, |c_1| = k+1, |c_2| = k,\]
\[|d_1| = -k-1, |d_2| = -k.\]
The differentials of the new generators are
\[\partial^{\varepsilon_n} c_1 = (-1)^{|c_2|}(nc_2 - c_2), \ \partial^{\varepsilon_n}c_2=0,\]
\[\partial^{\varepsilon_n}d_1 = 0, \ \partial^{\varepsilon_n}d_2 = d_1 - nd_1,\  \partial^{\varepsilon_n}a_1 = 0.\]
Thus, this generates $(1-n)$-torsion in the linearized Legendrian contact homology of the link, in degrees $k$ and $-k-1$. 
The homology of this link also has an extra generator in degree 1, represented by $a_3$.

We can now do a connected sum operation, as explained in section~\ref{sec:connsum}. This connected sum creates a new crossing $c$. For $k\neq1$, there are no disks of which the boundary goes through $c$ without turning in $c$, and for $k=1$, the only such disk must turn in $d_1$, which is sent to 0 by our augmentation. Due to the discussion which precedes Proposition~\ref{prop:connsum}, we may then extend $\varepsilon_n$ by setting $\varepsilon_n(c)=-1$. 

Since $\tau_1(a_3) = [\Lambda_{1,k}]$, $\rho_1$ is 
surjective, so that by Proposition~\ref{prop:connsum} this has the effect of removing one generator in degree 1 (in fact, since $\tau_1(a_1-a_2) =  [\Lambda_{\rm tr}]$, 
the remaining generator is $[a_1 - a_2 + a_3]$). The knot we obtain, which we shall call $\Lambda^T_k$, 
therefore has the following LCH:
$$
LCH^{\varepsilon_n}(\Lambda^T_k) = \Z[-1] \oplus \Z^2[0] \oplus \Z/(n-1)\Z[-k] \oplus \Z/(n-1)\Z[k+1].
$$

We can perform this operation using a collection of $r$ Legendrian unknots $\Lambda^T_{1,k_1}, \ldots, \Lambda^T_{r, k_r}$
with each $k_i < 0$. 

Denoting $\overline{k} = (k_1, \ldots, k_r)$, we therefore obtain a knot $\Lambda^T_{\overline{k}}$ which has the following 
linearized Legendrian contact homology:
\begin{align*}
LCH^{\varepsilon_n}(\Lambda^T_{\overline{k}}) = & \ \Z[-1] \oplus \Z^2[0] \\
& \oplus \bigoplus_{i=1}^r ( \Z/(n-1)\Z[-k_i] \oplus \Z/(n-1)\Z[k_i+1]).
\end{align*}

\subsection{Construction of the free part of linearized LCH}
\label{sec:freelinLCH}
We follow the construction used in \cite[Lemma~6.10]{BST} to establish the geography of generating family homology for Legendrian knots
with $\Z_2$ coefficients. To this end, we start with the standard Legendrian Hopf link, with a shift of $k \in \Z$ between the Maslov 
potentials of the two connected components of this link, denoted for this reason by $\Lambda^F_k$. Its Chekanov-Eliashberg DGA has 
two generators $a_1$ and $a_2$ of degree $1$, a generators $b_1$ of degree $k$ and a generator $b_2$ of degree $-k$. 
The differential is given by $\partial a_1 = b_1 b_2$ and $\partial a_2 = b_2 b_1$. We equip this link with the trivial augmentation
$\varepsilon_0$, which sends every generator to $0$. The linearized differential therefore vanishes, and 
$$
LCH^{\varepsilon_0}(\Lambda^F_k) = \Z^2[-1] \oplus \Z[-k] \oplus \Z[k].
$$
We then perform a Legendrian isotopy consisting of $k-1$ successive first Reidemeister moves in the front projection, in view of 
making a connected sum as shown in Figure~\ref{fig:Hopfconnsum}: we replace the contents of the dotted rectangle as in 
Figure~\ref{fig:connsum}.

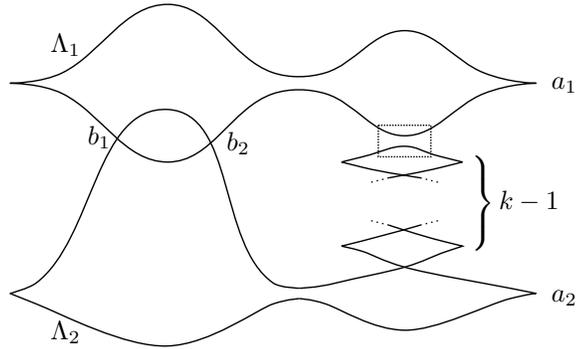
\begin{figure}
\labellist
\small\hair 2pt
\pinlabel {$\Lambda_1$} [bl] at 20 155
\pinlabel {$\Lambda_2$} [bl] at 20 0
\pinlabel {$b_1$} [bl] at 40 105
\pinlabel {$b_2$} [bl] at 115 102
\pinlabel {$a_1$} [bl] at 290 135
\pinlabel {$a_2$} [bl] at 290 20
\pinlabel {$\left.  \rule{0pt}{0.75cm} \right\} k-1$}  [bl] at 245 50
\endlabellist
  \centerline{\includegraphics[width=7cm]{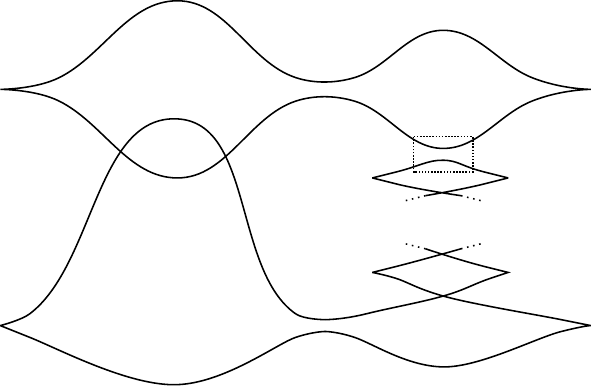}}
  \caption{Legendrian Hopf link before connected sum.}
  \label{fig:Hopfconnsum}
\end{figure}

Let us denote by $\Lambda_1$ the top component of $\Lambda^F_k$, by $\Lambda_2$ the bottom component of $\Lambda^F_k$,
and by $\widetilde{\Lambda}^F_k$ the Legendrian knot obtained by the connected sum operation on $\Lambda^F_k$. This connected sum creates a new crossing $c$, and like in the previous section, for $k\neq1$, there are no disks of which the boundary goes through $c$ without turning in $c$, and for $k=1$, the only such disk must turn in $b_1$, which is sent to 0 by our augmentation. So we may extend $\varepsilon_0$ to $\widetilde{\varepsilon}_0$ by setting $\widetilde{\varepsilon}_0(c)=-1$.
Since $\tau_1([a_i]) = [\Lambda_i]$ for $i = 1, 2$, Proposition~\ref{prop:connsum} then implies that 
$$
LCH^{\widetilde{\varepsilon}_0}(\widetilde{\Lambda}^F_k) = \Z[-1] \oplus \Z[-k] \oplus \Z[k].
$$

\subsection{Proof of Theorem~\ref{thm:linLCH}}

First, Corollary~\ref{cor:torsion_iso} and the discussion that follows shows that the torsion part of linearized LCH for any Legendrian knot
has the form $T \oplus \overline{T}[1]$ for some finitely generated graded torsion $\Z$-module $T$. Moreover, the remainder of 
Section~\ref{sec:constraints_lin} shows that the free part of linearized LCH for any Legendrian knot has the form $\Z[-1] \oplus F \oplus \overline{F}$
 for some finitely generated graded free $\Z$-module $F$. This proves the first part of Theorem~\ref{thm:linLCH}.

Conversely, given a finitely generated graded free $\Z$-module $F$, we denote the list of the gradings of its generators by 
$\overline{k} = (k_1, \ldots, k_r)$. Let $\widetilde{\Lambda}^F_{k_1}, \ldots, \widetilde{\Lambda}^F_{k_r}$ be the Legendrian knots
as in Section~\ref{sec:freelinLCH}, equipped with the augmentations so that their linearized LCH are isomorphic to $\Z[-1] \oplus \Z[-k_i] \oplus \Z[k_i]$,
for $i = 1, \ldots, r$. Let $\widetilde{\Lambda}^F_{\overline{k}}$ be the result of the connected sum of these Legendrian knots, with their front projections placed in disjoint rectangles of the $xz$-plane. As in the previous subsection, we may extend the augmentation to an augmentation $\widetilde{\varepsilon}_{\overline{k}}$ on the new knot by sending the new crossings to $-1$. Since
the $\tau_1$ map of each $\widetilde{\Lambda}^F_{k_i}$ is surjective, this implies by Proposition~\ref{prop:connsum} that the resulting linearized LCH is
given by
$$
LCH^{\varepsilon_0}(\widetilde{\Lambda}^F_{\overline{k}}) = \Z[-1] \oplus \bigoplus_{i=1}^r ( \Z[-k_i] \oplus \Z[k_i]) = \Z[-1] \oplus F \oplus \overline{F}.
$$

On the other hand, given a finitely generated graded torsion $\Z$-module $T$, we can decompose it as the direct sum of graded free 
$\Z/(n_i-1)\Z$-modules for $i = 1, \ldots, d$. We denote by $\overline{\ell}^i = (\ell^i_1, \ldots, \ell^i_{s_i})$ the gradings of the generators of these
free modules. For each $i=1, \ldots, d$, consider the Legendrian knot $\Lambda^T_{\overline{\ell}^i}$ equipped with its augmentation $\varepsilon_{n_i}$
as in Section~\ref{sec:torsionlin}. Let us perform the connected sum of these augmented Legendrian knots. The augmentation extends by sending the new crossing to $-1$. By Proposition~\ref{prop:connsum},  for each of these connected sums, we lose one generator in degree 1. This is due to the fact that the image of the alternate sum over all the cusps in $\Lambda_{\overline{\ell}_i}$ by $\tau_1$ is $[\Lambda_{\overline{\ell}_i}]$. We therefore obtain an augmented knot $(\Lambda_T, \varepsilon_T)$ whose 
linearized LCH is given by
\begin{align*}
LCH^{\varepsilon_T}(\Lambda_T) &= \Z[-1] \\
& \oplus \bigoplus_{i=1}^d \left(\Z^{2}[0] \oplus \bigoplus_{j=1}^{s_i} ( \Z/(n_i-1)\Z[-\ell_j] \oplus \Z/(n_i-1)\Z[\ell_j+1]) \right) \\
&= \Z[-1] \oplus \Z^{2d}[0] \oplus T \oplus \overline{T}[1].
\end{align*}

It now suffices to make the connected sum of the augmented Legendrian knots $(\widetilde{\Lambda}^F_{\overline{k}}, \varepsilon_0)$ and
$(\Lambda_T, \varepsilon_T)$, to which the augmentation extends by sending the new crossing to $-1$ as in the discussion before Proposition~\ref{prop:connsum}. Applying Proposition~\ref{prop:connsum} one last time finishes the proof of Theorem~\ref{thm:linLCH}.

\section{Torsion in bilinearized Legendrian contact homology}
\label{sec:torsion_bilin}

\subsection{Bilinearized Legendrian contact homology of the trefoil knot}\label{subsec:trefoil_bi}
Let us compute the bilinearized Legendrian contact homology of the knot we studied in Section \ref{sec:torsion_lin}, when we take two non-dga-homotopic augmentations $\varepsilon_1$ and $\varepsilon_2$. 

We start by doing the computations for the trefoil knot $\Lambda_{\rm tr}$. The augmentations must satisfy, for $i=1,2$, $$1+\varepsilon_i(b_1)+\varepsilon_i(b_3)+\varepsilon_i(b_1b_2b_3)=0,$$ and $$1-\varepsilon_i(b_1t)-\varepsilon_i(b_3t)+\varepsilon_i(b_3b_2b_1t)=0.$$
We must therefore once again set $\varepsilon_i(t)=\varepsilon_i(t^{-1})=-1$, and we can make various choices for $b_1$, $b_2$ and $b_3$.

\begin{Ex}  \label{ex:tau0-bilin}
Let us choose the following augmentations for the trefoil knot:
\begin{align*}
\varepsilon_1(b_1)&=-1 & \varepsilon_2(b_1)&=-1,\\
\varepsilon_1(b_2)&=0 & \varepsilon_2(b_2)&=n\in\Z,\\
\varepsilon_1(b_3)&=0 & \varepsilon_2(b_3)&=0.
\end{align*}
The bilinearized differential is then given by
\begin{align*}
\partial^{\varepsilon_1,\varepsilon_2} b_1 &= \partial^{\varepsilon_1,\varepsilon_2} b_2 = \partial^{\varepsilon_1,\varepsilon_2} b_3 = 0, \\
\partial^{\varepsilon_1,\varepsilon_2} a_1 &= b_1 + b_3,\\
 \partial^{\varepsilon_1,\varepsilon_2} a_2 &=b_1 + b_3 -nb_3.
 \end{align*}
Hence, $[b_1] = -[b_3]$ and $n[b_3]=0$ so that
$$
LCH^{\varepsilon_1, \varepsilon_2}(\Lambda_{\rm tr}) = \Z[0] \oplus \Z/n\Z[0].
$$
Unlike in linearized LCH, we obtain a single torsion generator. Let us also compute the map 
$\tau_0: LCH^{\varepsilon_1, \varepsilon_2}_0(\Lambda_0) \to H_0(\Lambda_0) = \Z$, 
which is given by $\varepsilon_2 - \varepsilon_1$ at chain level: $\tau_0(b_1) = 0$, $\tau_0(b_2) = n$ and $\tau_0(b_3) = 0$. In particular,
the image of $\tau_0$ in $n\Z \subset H_0(\Lambda_0)$, so that the map $\tau_0$ is not surjective when $|n| > 1$. 
This shows that the statement of Proposition~\ref{prop:tau} cannot be improved.
\end{Ex}

In order to construct torsion in any grading, we now make another choice for the augmentations $\varepsilon_1$ and $\varepsilon_2$:
\begin{align*}
\varepsilon_1(b_1)&=0 & \varepsilon_2(b_1)&=-1,\\
\varepsilon_1(b_2)&=0 & \varepsilon_2(b_2)&=n\in\Z,\\
\varepsilon_1(b_3)&=-1 & \varepsilon_2(b_3)&=0.
\end{align*}
The differentials on the trefoil knot become 
\begin{align*}
\partial^{\varepsilon_1,\varepsilon_2} b_1 &= \partial^{\varepsilon_1,\varepsilon_2} b_2 = \partial^{\varepsilon_1,\varepsilon_2} b_3 = 0, \\
\partial^{\varepsilon_1,\varepsilon_2} a_1 &= b_1 + b_3,\\
 \partial^{\varepsilon_1,\varepsilon_2} a_2 &=b_1 + b_3 + b_2 -nb_3.\end{align*}
 We therefore no longer have any homology generators in degree 1, and we have a unique homology generator in degree 0, and so
 $$
 LCH^{\varepsilon_1, \varepsilon_2}(\Lambda_{\rm tr}) = \Z[0].
 $$
 
Now if we link this knot with a Legendrian unknot $\Lambda_{1,k}$ like we did in Figure \ref{fig:trefle_trivial}, with a shifted Maslov potential of $k\in\Z$, we again have five new generators $c_1$, $c_2$, $d_1$, $d_2$, and $a_3$ over which both augmentations extend to 0, and of gradings
 \[|a_3| = 1, |c_1| = k+1, |c_2| = k,\]
\[|d_1| = -k-1, |d_2| = -k,\]
which have the linearized differentials
\[\partial^{\varepsilon} c_1 = (-1)^{|c_2|}(nc_2 - c_2), \ \partial^{\varepsilon}c_2=0,\]
\[\partial^{\varepsilon}d_1 = 0, \ \partial^{\varepsilon}d_2 = d_1,\  \partial^{\varepsilon}a_1 = 0.\]
So $d_1$ and $d_2$ cancel out in homology, and we are left with a unique free generator $a_3$ in degree 1, and a unique torsion generator $c_2$ in degree $k$. 

Since $\tau_1([a_3])=[\Lambda_{1,k}]$, the fundamental class of $\Lambda_{1,k}$, $\rho_1$ is surjective for this link, and so by Proposition~\ref{prop:connsum} when we do the connected sum between $\Lambda_0$ and $\Lambda_{1,k}$ (for which the augmentation extends by sending the new crossing to $-1$), we simply lose a generator in degree 1. 
Therefore, we obtain a knot $\Lambda^T_k$ such that
$$
LCH^{\varepsilon_1, \varepsilon_2}(\Lambda_k) = LCH^{\varepsilon_1, \varepsilon_2}(\Lambda_{\rm tr}) \oplus \Z/(n-1)\Z[-k]
= \Z[0] \oplus \Z/(n-1)\Z[-k].
$$

By repeating this operation with a series of different Legendrian unknots $\Lambda_{1, k_1}, \ldots, \Lambda_{r,k_r}$ and denoting
$\overline{k} = (k_1, \ldots, k_r)$, we get a knot $\Lambda_{\overline{k}}$, which has the following bilinearized Legendrian 
contact homology 

$$
LCH^{\varepsilon_1,\varepsilon_2}(\Lambda^T_{\overline{k}}) = \Z[0] \oplus \bigoplus_{i=1}^r \Z/(n-1)\Z[-k_i].
$$
This shows that there exists a knot whose bilinearized Legendrian contact homology has a minimal free part (one unique degree 0 generator) 
and with torsion in any degree. In order to combine different torsion orders, we cannot use different trefoils and then make a connected sum, 
because the assumptions of Proposition~\ref{prop:connsum} will not be satisfied. Instead, we need to replace the trefoil with a more complicated 
Legendrian knot, and this will be done in the next section.

\subsection{Constructions involving the $2n$-copy of the Legendrian unknot}\label{subsection:2n-copy}
\subsubsection{Bilinearized Legendrian contact homology of the link}\label{subsubsection:link_homology}
We study here the bilinearized Legendrian contact homology of $\Lambda^{(2n)}$, the $2n$-copy of the Legendrian unknot, where $n$ is a positive integer. This is a Legendrian link formed of $2n$ copies of the standard Legendrian unknot that are slightly pushed off from one another so that every component still intersects every other component. We denote these components by $\Lambda_1, \ldots, \Lambda_{2n}$ from bottom to top. We also set the Maslov potentials of these components so that the Maslov potential of $\Lambda_{i+1}$ is one higher than the Maslov potential of $\Lambda_i$, for
$i = 1, \ldots, 2n-1$. We put a base point on the lower strand of each component. See Figure \ref{fig:6-copy}. 

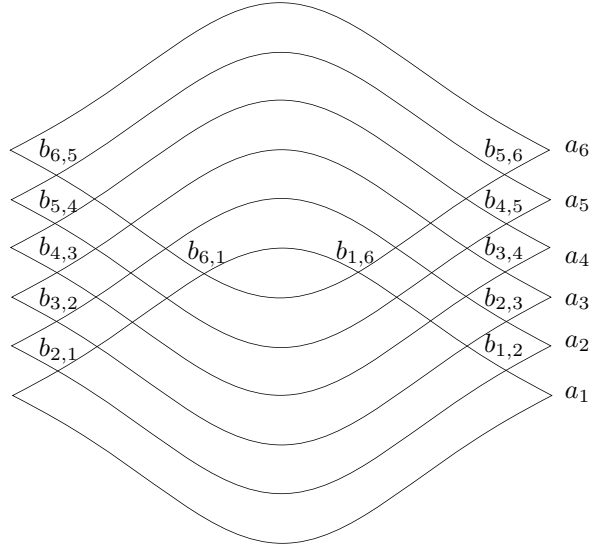
\begin{figure}[h]
\labellist
\small\hair 2pt
\pinlabel {$a_6$} [bl] at 137 94
\pinlabel {$a_5$} [bl] at 137 81
\pinlabel {$a_4$} [bl] at 137 68
\pinlabel {$a_3$} [bl] at 137 58
\pinlabel {$a_2$} [bl] at 137 48
\pinlabel {$a_1$} [bl] at 137 36
\pinlabel {$b_{1,2}$} [bl] at 118 46
\pinlabel {$b_{2,3}$} [bl] at 118 57
\pinlabel {$b_{3,4}$} [bl] at 118 69
\pinlabel {$b_{4,5}$} [bl] at 118 80
\pinlabel {$b_{5,6}$} [bl] at 118 92
\pinlabel {$b_{2,1}$} [bl] at 13 45
\pinlabel {$b_{3,2}$} [bl] at 13 57
\pinlabel {$b_{4,3}$} [bl] at 13 69
\pinlabel {$b_{5,4}$} [bl] at 13 80
\pinlabel {$b_{6,5}$} [bl] at 13 92
\pinlabel {$b_{6,1}$} [bl] at 48 68
\pinlabel {$b_{1,6}$} [bl] at 83 68
\endlabellist
  \centerline{\includegraphics[width=8cm]{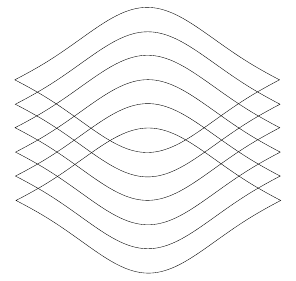}}
  \caption{6-copy of the Legendrian unknot.}
  \label{fig:6-copy}
\end{figure}

This gives us $2n$ cusps $a_1, …, a_{2n}$, as well as crossings $b_{i,j}$ with $i \neq j$ for $i,j = 1, \ldots, 2n$. The crossings are named such that $b_{i,j}$ and $b_{j,i}$ are the two only intersection points between the $\Lambda_i$ and $\Lambda_j$, and $b_{i,j}$ is the one for which the strand with the more negative slope belongs to $\Lambda_i$. 

The gradings are the following:
\[ | a_i | =  1, |b_{i,j}| = i-j +1 \text{ if $i<j$}, |b_{i,j}| = i - j - 1 \text{ if $i>j$}. \]

Due to Leverson's result~\cite{L}, we once again send $t_i$ for $i=1,…,2n$ to $-1$.

The only generators with grading zero, and therefore the only generators which can be augmented, are the intersections between two consecutive unknots. Since the differential of $b_{i,i-2}$ is $b_{i, i-1}b_{i-1,i-2}$, for any augmentation $\varepsilon$ and any $i$, we must either have $\varepsilon(b_{i, i-1})=0$ or $\varepsilon(b_{i-1,i-2})=0$. Also, to cancel $\varepsilon\circ\partial (a_i)$, we must have $\varepsilon(b_{i,i-1}b_{i-1,i})=0$.

Let us note that, for any choice of augmentations $\varepsilon_1$ and $\varepsilon_2$, the bilinearized differentials are 
\begin{align*} \partial^{\varepsilon_1, \varepsilon_2} a_i &= \varepsilon_1(b_{i, i-1})b_{i-1,i} + \varepsilon_1(b_{i,i+1})b_{i+1, i} + \varepsilon_2(b_{i-1,i})b_{i, i-1} + \varepsilon_2(b_{i+1, i})b_{i, i+1},\\
\partial^{\varepsilon_1, \varepsilon_2} b_{i,j} &= \varepsilon_1(b_{i, i-1})b_{i-1,j} +(-1)^{i+j} \varepsilon_2(b_{j+1, j})b_{i, j+1}.\end{align*}
We make the following choice of augmentations:
\begin{align*}\varepsilon_1(b_{2i, 2i-1})&=1 & \varepsilon_2(b_{2i, 2i-1})&=0 ,\\
\varepsilon_1(b_{2i+1, 2i}) &= 0 & \varepsilon_2(b_{2i+1, 2i}) &= 1,\\
\varepsilon_1(b_{2i-1, 2i}) &= 0 & \varepsilon_2(b_{2i-1, 2i}) &= 0,\\
\varepsilon_1(b_{2i, 2i+1}) &= 0 & \varepsilon_2(b_{2i, 2i+1}) &= 0.\end{align*}
So, the differentials are now 
\begin{equation}\label{eq:computation_diff}\begin{aligned} 
\partial^{\varepsilon_1, \varepsilon_2} a_{2i} &= b_{2i,2i+1} + b_{2i-1,2i}& \partial^{\varepsilon_1, \varepsilon_2}a_{2i+1} &= 0, \\
\partial^{\varepsilon_1, \varepsilon_2} b_{2i,2j} &= b_{2i-1,2j} + b_{2i, 2j+1} & \partial^{\varepsilon_1, \varepsilon_2} b_{2i+1,2j+1} &= 0, \\
 \partial^{\varepsilon_1, \varepsilon_2} b_{2i+1,2j} &= b_{2i+1,2j+1}& \partial^{\varepsilon_1, \varepsilon_2} b_{2i,2j+1} &= b_{2i-1,2j+1},
\end{aligned}\end{equation}
(when such terms do not exist, they are replaced by 0).

Let us see what generators survive in homology. Firstly, generators of the form $b_{2i+1,2j+1}$ cancel out in homology because they are in the image of the differential: $\partial^{\varepsilon_1, \varepsilon_2} b_{2i+2,2j+1} = b_{2i+1,2j+1}$. Also, if $b_{2i+1,2j+1}$ exists (so if $i\neq$j and $2i+1$ and $2j+1$ are both between 1 and $2n$), then $b_{2i+2,2j+1}$ always exists, because $2n$ is even. 

In terms of linear combinations that the differential can send to 0, we have $\partial^{\varepsilon_1, \varepsilon_2} (b_{2i,2j+1} + b_{2i-1,2j})=0$. However, this term is the image by the differential of $b_{2i,2j}$, so cancels out in homology. No other non-trivial linear combinations can be cancelled by the differential.

We also have $\partial^{\varepsilon_1, \varepsilon_2}a_{2i+1} = 0$, and these terms never appear in the image of the differential, so this gives us $n$ homology generators in degree 1.

Finally, the last case to consider is when the terms that could appear in the differential actually do not exist. This is the case for elements of the form $b_{i,i-1}$: we have $\partial^{\varepsilon_1, \varepsilon_2} b_{2i,2i-1} = b_{2i-1,2i-1}$, which does not exist, and we have $\partial^{\varepsilon_1, \varepsilon_2} b_{2i+1,2i} = -b_{2i+1,2i+1}$, which also is a term that does not exist. So the differentials are 0, and we have $\partial^{\varepsilon_1, \varepsilon_2} b_{2i, 2i-2} = b_{2i-1,2i-2} + b_{2i, 2i-1}$. Therefore, this produces $n$ homology generators, and since elements of the form $b_{i,i-1}$ are of grading 0, we get $n$ homology generators in degree 0.

We therefore have:
 $$
 LCH^{\varepsilon_1,\varepsilon_2}(\Lambda^{(2n)}) = \Z^n[0] \oplus \Z^n[-1].
 $$

\subsubsection{A special connected sum}\label{subsection:unclasp}
We now want to do some sort of connected sum to obtain a knot from $\Lambda^{(2n)}$. Unfortunately, the usual connected sum described in section~\ref{sec:connsum} can only be done $n-1$ times before it starts adding generators in degree 0, and we need to do it $2n-1$ times. Nevertheless, we still proceed with this operation: we connect all components $\Lambda_{2i-1}$ for $i = 1, \ldots, n$ to each other and all components $\Lambda_{2i}$ for $i = 1, \ldots, n$ to each other. Each time, there are no disks which appear after the connected sum which go through the new crossing created without turning in it, and so due to the discussion which precedes Proposition~\ref{prop:connsum}, our augmentations extend to augmentations that we name $\bar\varepsilon_1$ and $\bar\varepsilon_2$ by sending the new crossings to $-1$. Since the homology in degree 1 is generated by the odd cusps and $\tau_1([a_{2i-1}]) = [\Lambda_{2i-1}]$, the map $\rho_1$ is surjective and by Proposition~\ref{prop:connsum} each connected sum between odd-numbered components reduces the homology by one degree 1 generator. Since we do this operation $n-1$ times, we end up with one generator in degree 1,  $[a_{\text{odd}}]$, represented by an alternate sum of the odd cusps. On the other hand, the even cusps are null in homology, so that the map $\rho_1$ vanishes and by Proposition~\ref{prop:connsum} each connected sum between even-numbered components gives us an additional degree 0 generator. Again, we do this $n-1$ times, so we end up with $2n-1$ degree 0 generators. We obtain a link $\Lambda^{(2n)}_{\rm e,o}$ which has two connected components, $\Lambda_{\rm even}$ and $\Lambda_{\rm odd}$, and which has the homology
 $$
 LCH^{\bar\varepsilon_1,\bar\varepsilon_2}(\Lambda^{(2n)}_{\rm e,o}) = \Z^{2n-1}[0] \oplus \Z[-1].
 $$
 More precisely, we perform these connected sums in the following way: for each $i= 1, \dots, n-1$, we perform the two-copy of a first Reidemeister move
 on $\Lambda_{2i-1}$ and $\Lambda_{2i}$. We then perform twice a second Reidemeister move so that the right cusp of $\Lambda_{2i+1}$ crosses
 twice the uppermost strand of $\Lambda_{2i}$. We are now in position to make a connected sum between $\Lambda_{2i-1}$ and $\Lambda_{2i+1}$
 as well as a connected sum between $\Lambda_{2i}$ and $\Lambda_{2i+2}$, as shown on the left of Figure~\ref{fig:double-conn-sum}.
Finally, we perform a third Reidemeister move so that the top stratum of $\Lambda_{2i+1}$ passes above the newly created crossing due to 
the connected sum of $\Lambda_{2i}$ and $\Lambda_{2i+2}$. The right of Figure~\ref{fig:double-conn-sum} illustrates as a dotted line of the
top stratum of $\Lambda_{2i+1}$ after this third Reidemeister move.

 \begin{figure}[h]
  \labellist
\small\hair 2pt
\pinlabel {$\Lambda_{2i-1}$} [bl] at -25 5
\pinlabel {$\Lambda_{2i}$} [bl] at -25 30
\pinlabel {$\Lambda_{2i+1}$} [bl] at -25 130
\pinlabel {$\Lambda_{2i+2}$} [bl] at -25 175
\pinlabel {\scriptsize$\Lambda_{2i-1} \# \Lambda_{2i+1}$} [bl] at 300 -5
\pinlabel {\scriptsize$\Lambda_{2i}\#\Lambda_{2i+2}$} [bl] at 255 48
\pinlabel {\scriptsize$\Lambda_{2i-1} \# \Lambda_{2i+1}$} [bl] at 450 150
\pinlabel {\scriptsize$\Lambda_{2i}\#\Lambda_{2i+2}$} [bl] at 460 210
\endlabellist
 \includegraphics[width=5cm]{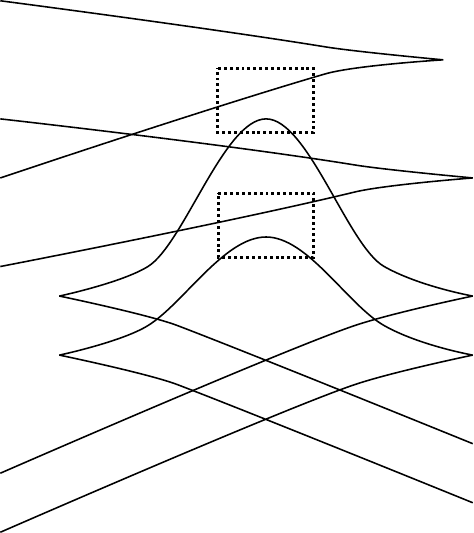} \hspace{8mm}
 \includegraphics[width=5cm]{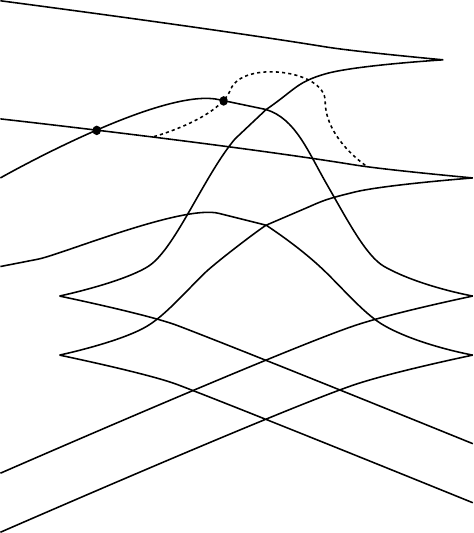}
  \caption{Before and after the connected sums between $\Lambda_{2i-1}$ and $\Lambda_{2i+1}$ and between $\Lambda_{2i}$ and $\Lambda_{2i+2}$.}
  \label{fig:double-conn-sum}
\end{figure}
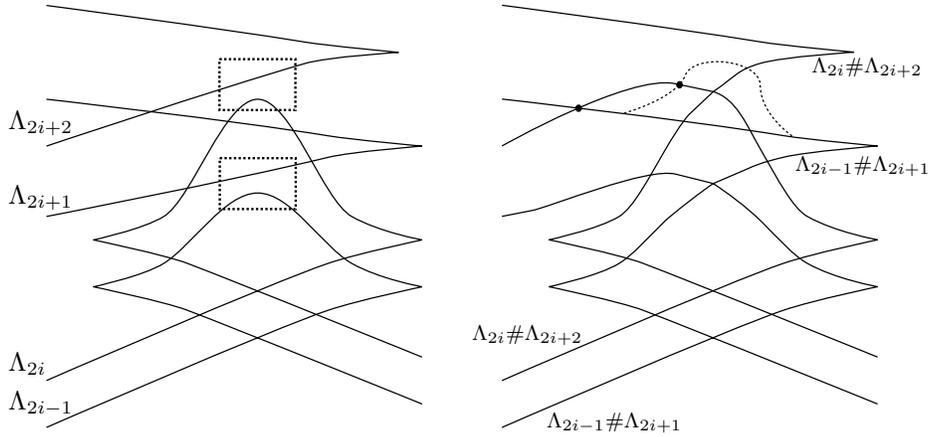
 
Consider now the pair of degree 0 crossings depicted as black dots on the right of Figure~\ref{fig:double-conn-sum} and corresponding to 
intersection points of $\Lambda_{2i-1} \# \Lambda_{2i+1}$ and $\Lambda_{2i}\#\Lambda_{2i+2}$ just to the left of the newly created crossing 
due to the connected sum of $\Lambda_{2i}$ and $\Lambda_{2i+2}$.
In order to get rid of this pair of crossings, we perform an operation known as an \emph{unclasp}: that is, we get rid of two crossings by raising the band that passes underneath (and/or lowering the one that passes above). We want to show that each 
unclasp, for $i = 1, \ldots, n-1$, has the effect of removing two generators of degree 0 in homology. 

We first need to compute the augmentations on the new generators after the double connected sum, and in particular check that both crossings which are removed are sent to 0. We will therefore describe the effect of a double connected sum on the augmentations, step by step. To distinguish the augmentations after the double connected sum with the ones before, we will call the new augmentations $\bar\varepsilon_1$ and $\bar\varepsilon_2$.

We start by doing a two-copy of a Reidemeister I operation. This move is in fact nothing more than two Reidemeister I moves, which create two new right cusps $a_{\rm I,2i-1}$ and $a_{\rm I,2i}$, as well as two new crossings $b_{\rm I,2i-1}$ and $b_{\rm I,2i}$, followed by two Reidemeister II moves which make the left and right cusps of the bottom component cross the strands of the top component and create four new cusps $b_{\rm II,1}$ to $b_{\rm II,4}$, followed by a Reidemeister III move which brings the strand that goes from $b_{\rm II,2}$ to $b_{\rm II,4}$ over $b_{\rm I,2i}$. See Figure \ref{fig:double_reidemeister_1}.

It is not hard to see that after the first pair of Reidemeister I moves, $b_{\rm I,2i-1}$ and $b_{\rm I,2i}$ must be sent to $-1$ by either augmentation. 
For the Reidemeister II moves, we will use the pullback of the augmentations by the DGA map $\Phi_{\rm II}$ associated to this operation. The map is computed in \cite[section 6.3.3.]{EKK}: if $a$ and $b$ are the two crossings created by a Reidemeister II operation, ordered by their grading, and $\partial a = \pm b+v$, then we have \begin{equation} \Phi_{\rm II}(a)=0,\ \Phi_{\rm II}(b)=v,\end{equation} which therefore gives us, for $j=1,2$, \begin{equation}\label{eq:augmentations_reidemeister_II}\bar\varepsilon_j(a)=0, \ \bar\varepsilon_j(b)=\mp \varepsilon_j(v).\end{equation} In our case, we start by doing the first Reidemeister II move on the right, and we have $\partial b_{\rm II,1}=-b_{\rm II,2}b_{\rm I,2i-1}-b_{\rm I,2i}b_{2i,2i-1}$ (there may be some other terms, but all which are sent to 0 by both augmentations, so we ignore them). We may then deduce that $\bar\varepsilon_j(b_{\rm II,2})=\varepsilon_j(b_{\rm I,2i}b_{2i,2i-1})=-\varepsilon_j(b_{2i,2i-1})$. The Reidemeister II move on the left gives us $\partial b_{\rm II,4}=b_{\rm II,3}$, and so both crossings are sent to 0 by both augmentations.
Finally, the DGA map $\Phi_{\rm III}$ for the Reidemeister III move gives us 
\begin{equation}\label{eq:reidemeister_III}\begin{aligned} \Phi_{\rm III}(b_{\rm II,2})&=b_{\rm II,2}+b_{\rm I,2i}b_{\rm II,4},\\ \Phi_{\rm III}(b_{\rm I,2i})&=b_{\rm I,2i},\\ \Phi_{\rm III}(b_{\rm II,4})&=b_{\rm II,4},\end{aligned}\end{equation} per \cite[section 6.3.2.]{EKK}.  The augmentations therefore stay the same, except for $b_{\rm II,2}$, for which we get $\bar\varepsilon_j(b_{\rm II,2})=\varepsilon_j(b_{\rm II,2})+\varepsilon_j(b_{\rm I,2i}b_{\rm II,4})$. The value of the augmentations of $b_{\rm II,2}$ therefore also stay unchanged, as $\varepsilon_j(b_{\rm II,4})=0$. 

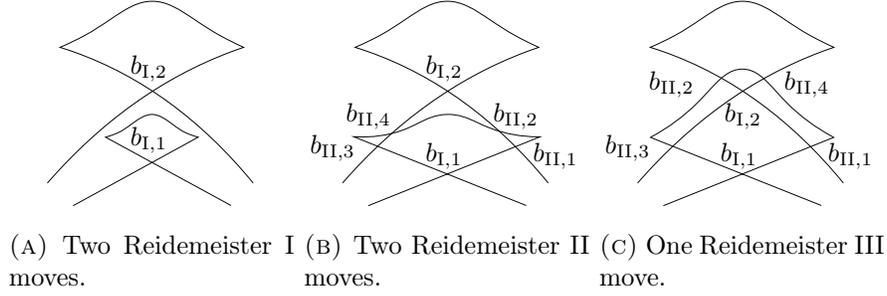
\begin{figure}[h]
\begin{subfigure}{0.3\textwidth}
\labellist
\small\hair 2pt
\pinlabel {$b_{\rm I,1}$} [bl] at 110 75
\pinlabel {$b_{\rm I,2}$} [bl] at 110 165
\endlabellist
 \centerline{\includegraphics[width=3cm,height=3cm]{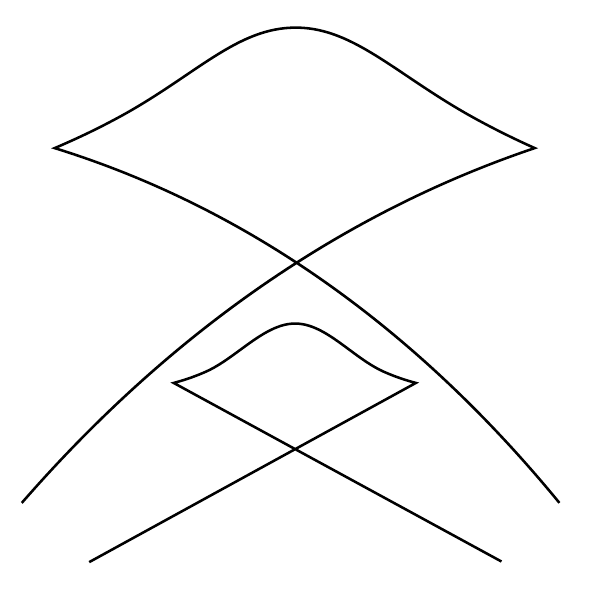}}
  \caption{Two Reidemeister I moves.}
  \label{fig:Reidemeister_I_double}
  \end{subfigure}
\begin{subfigure}{0.3\textwidth}
\labellist
\small\hair 2pt
\pinlabel {$b_{\rm I,1}$} [bl] at 110 55
\pinlabel {$b_{\rm I,2}$} [bl] at 110 165
\pinlabel {$b_{\rm II,1}$} [bl] at 245 55
\pinlabel {$b_{\rm II,2}$} [bl] at 195 105
\pinlabel {$b_{\rm II,3}$} [bl] at -35 65
\pinlabel {$b_{\rm II,4}$} [bl] at 10 105
\endlabellist
  \centerline{\includegraphics[width=3cm,height=3cm]{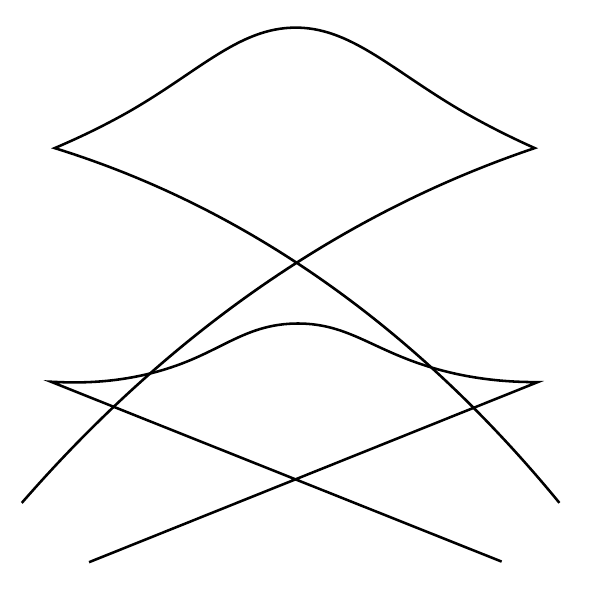}}
  \caption{Two Reidemeister II moves.}
  \label{fig:Reidemeister_II_double}
\end{subfigure}
\begin{subfigure}{0.3\textwidth}
\labellist
\small\hair 2pt
\pinlabel {$b_{\rm I,1}$} [bl] at 110 55
\pinlabel {$b_{\rm I,2}$} [bl] at 115 105
\pinlabel {$b_{\rm II,1}$} [bl] at 245 55
\pinlabel {$b_{\rm II,3}$} [bl] at -35 65
\pinlabel {$b_{\rm II,2}$} [bl] at 20 145
\pinlabel {$b_{\rm II,4}$} [bl] at 190 145
\endlabellist
  \centerline{\includegraphics[width=3cm,height=3cm]{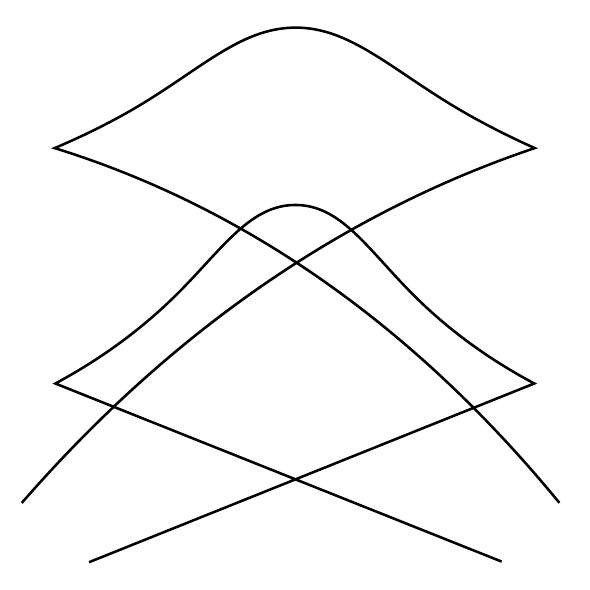}}
  \caption{One Reidemeister III move.}
  \label{fig:Reidemeister_III_double}
\end{subfigure}
  \caption{Two-copy of the Reidemeister I operation.}
  \label{fig:double_reidemeister_1}
\end{figure}

The following operations are seen in Figure \ref{fig:pre_unclasp}. We start with two Reidemeister II moves. Each one creates two new crossings, which we label $b_{\rm II,5}, b_{\rm II,6}$ and $b_{\rm II,7}, b_{\rm II,8}$. The differentials are of the form \begin{equation*}\partial b_{\rm II,5}=b_{\rm II,6}(1+c), \ \partial b_{\rm II,8}= b_{\rm II,7}(1+c)\end{equation*} where $c$ is a sum of elements of the form $b_{i,j}b_{j,i}$, and therefore sent to 0 by either augmentation. Per equation \eqref{eq:augmentations_reidemeister_II}, the values of the augmentations of these new crossings must be 0. 

Next we perform the two connected sums, which each create one new crossing, which we label $b_{\rm o}$ for the crossing between $\Lambda_{2i-1}$ and $\Lambda_{2i+1}$, and $b_{\rm e}$ for the crossing between $\Lambda_{2i}$ and $\Lambda_{2i+2}$. The previous augmentations extend by sending the new crossings to $-1$.

Finally, the last move is a Reidemeister III move which brings the strand passing through $b_{\rm II,6}$ and $b_{\rm II,8}$ over $b_{\rm e}$. Per equation \ref{eq:reidemeister_III}, we get \begin{gather*}\Phi_{\rm III}(c)= c, \ c\neq b_{\rm II,6},\\ \Phi_{\rm III}(b_{\rm II,6})=b_{\rm II,6}\pm b_{\rm e}b_{\rm II,8}.\end{gather*} Since $b_{\rm e}$ is sent to $-1$ by both augmentations and $b_{\rm II,8}$ and $b_{\rm II,6}$ are sent to 0, the pullback of the augmentations by this map tells us that the values of the augmentations over every generator are unchanged. 

To summarize, we get that the values of the augmentations on the generators which existed before the unclasp remain unchanged, and the others take the following values:
\begin{align*}
\bar\varepsilon_1(b)=\bar\varepsilon_2(b)&=  0, \text{ if $b\neq b_{\rm I,2i-1},b_{\rm I,2i},b_{\rm II,2},b_{\rm e},b_{\rm o}$},\\
\bar\varepsilon_1(b)=\bar\varepsilon_1(b)&=-1, \text{ if $b= b_{\rm I,2i-1},b_{\rm I,2i},b_{\rm e},b_{\rm o}$},\\
\bar\varepsilon_j(b_{\rm II,2})=-\varepsilon_j(b_{2i,2i-1})&=\begin{cases} -1 &\text{ if $j=1$},\\ 0 &\text{ if $j=2$}.\end{cases}
\end{align*}
In particular, $\bar\varepsilon_j(b_{\rm II,6})=0$, for $j=1,2$.

\begin{figure}[h]
\begin{subfigure}{0.4\textwidth}
\labellist
\small\hair 2pt
\endlabellist
 \centerline{\includegraphics[width=6cm,height=4.5cm]{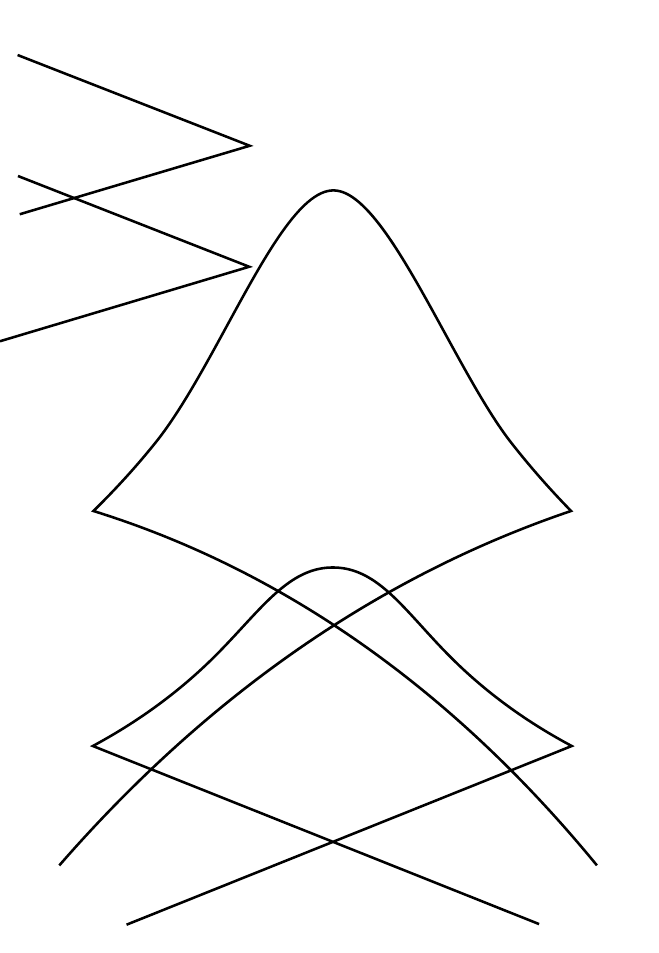}}
  \caption{After the double Reidemeister I move.}
  \label{fig:Reidemeister_II_before}
  \end{subfigure}
\begin{subfigure}{0.4\textwidth}
\labellist
\small\hair 2pt
\pinlabel {$b_{\rm II,5}$} [bl] at 205 275
\pinlabel {$b_{\rm II,6}$} [bl] at 175 335
\pinlabel {$b_{\rm II,7}$} [bl] at 50 260
\pinlabel {$b_{\rm II,8}$} [bl] at 75 315
\endlabellist
  \centerline{\includegraphics[width=6cm,height=4.5cm]{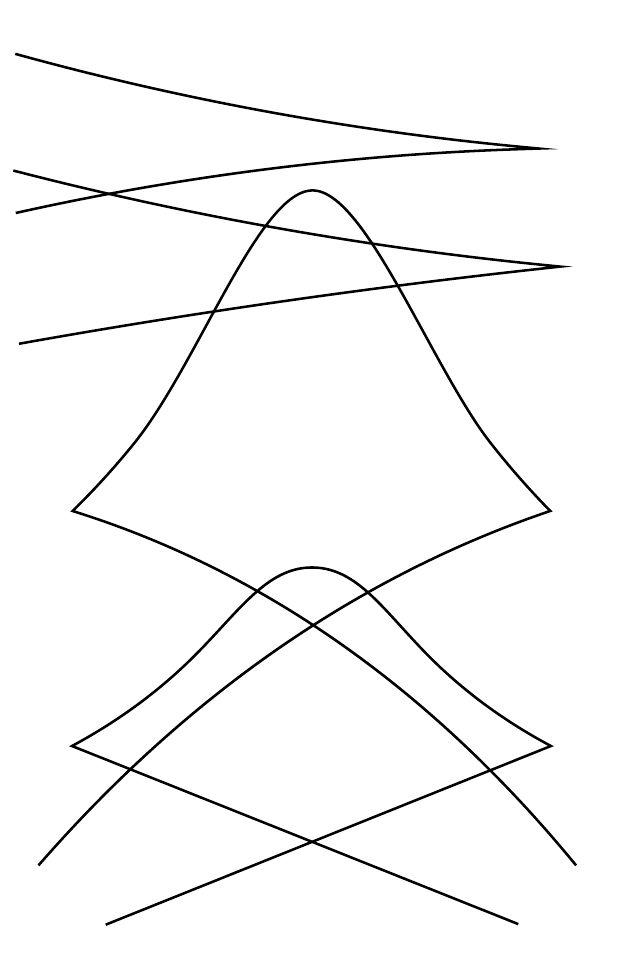}}
  \caption{Two Reidemeister II moves.}
  \label{fig:Reidemeister_II_twice}
\end{subfigure}
\begin{subfigure}{0.4\textwidth}
\labellist
\small\hair 2pt
\pinlabel {$b_{\rm e}$} [bl] at 135 240
\pinlabel {$b_{\rm o}$} [bl] at 150 342
\endlabellist
  \centerline{\includegraphics[width=6cm,height=4.5cm]{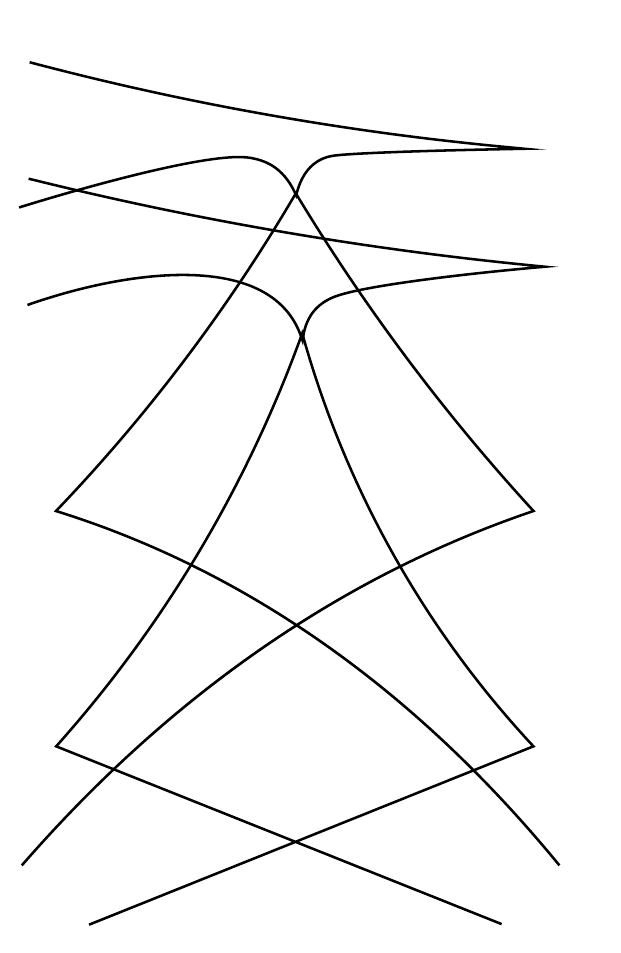}}
  \caption{Double connected sum.}
  \label{fig:double_somme_co}
\end{subfigure}
\begin{subfigure}{0.4\textwidth}
\labellist
\small\hair 2pt
\pinlabel {$b_{\rm II,8}$} [bl] at 180 340
\pinlabel {$b_{\rm II,6}$} [bl] at 85 345
\pinlabel {$b_{\rm o}$} [bl] at 140 310
\pinlabel {$b_{2i+1,2i+2}$} [bl] at 00 375
\endlabellist
  \centerline{\includegraphics[width=6cm,height=4.5cm]{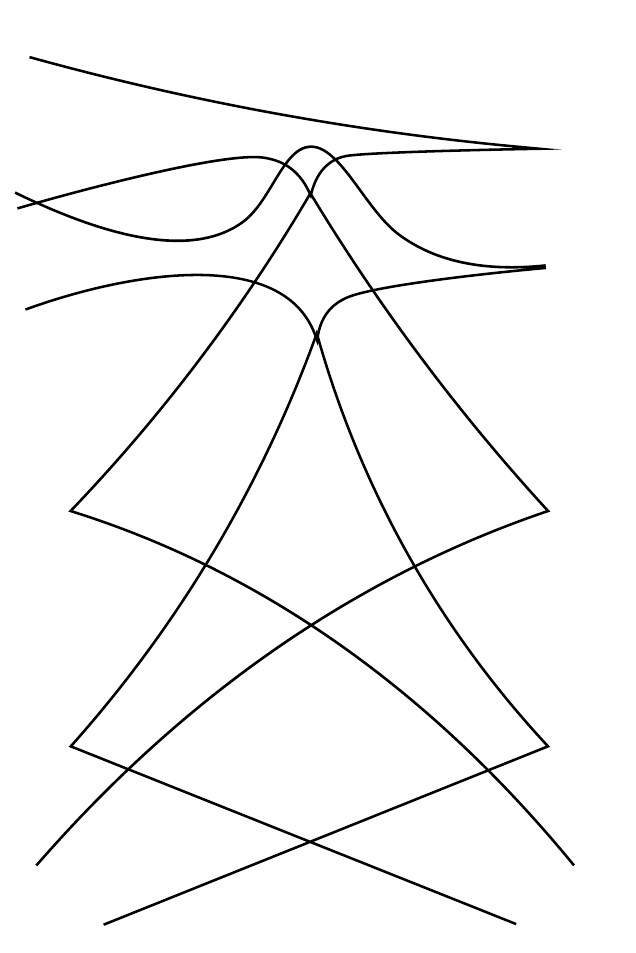}}
  \caption{One Reidemeister III move.}
  \label{fig:reidemeister_III}
\end{subfigure}
  \caption{The double connected sum operation.}
  \label{fig:pre_unclasp}
\end{figure}
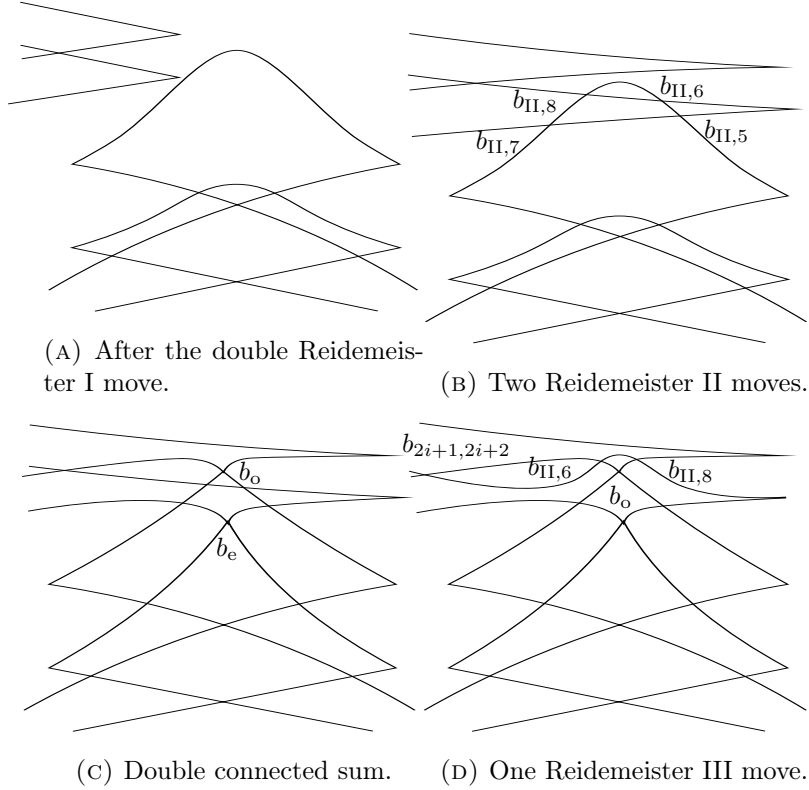

The following lemma now tells us that this choice of augmentations can still be made after the unclasp.
\begin{Lem}\label{lem:augmentations_unclasp} Let $\bar\varepsilon_1$, $\bar\varepsilon_1$ be the augmentations computed after the double connected sum, before the unclasp. Then their values on the cusps not removed by the unclasp are still valid after the unclasp operation.
\end{Lem} 
\begin{proof} The two crossings which are removed with the unclasp are $b_{\rm II,6}$ and $b_{2i+1,2i+2}$. We want to check that, after their removal, the equalities $\bar\varepsilon_j\circ\partial=0$ still hold for $j=1,2$. Since they are both sent to 0 by our augmentations, the removal of disks which turn in $b_{\rm II,6}$ or $b_{2i+1,2i+2}$ does not affect this equality. We must therefore only check that there can be no disks which appear after the unclasp which change this equality. In fact, the disks which appear after the unclasp correspond exactly to the disks which contribute to the differential of $b_{2i+1,2i+2}$: see Figure \ref{fig:disks_unclasp}. After the unclasp, the term which was $\partial b_{2i+1,2i+2}$ is added to the differentials of the crossings on the right of the unclasp (perhaps multiplied by other terms). Then the equality $\bar\varepsilon_j\circ\partial b_{2i+1,2i+2}=0$ tells us that the disks which appear after the unclasp cannot change the equality $\bar\varepsilon_j\circ\partial=0$ after the unclasp, as they will just add 0 on the right side of the equation.
\end{proof}

We can now describe the effect of the unclasp on the homology.

\begin{Prop}\label{prop:unclasp_1}
Let $\bar\varepsilon_1$, $\bar\varepsilon_2$ be the augmentations after the unclasp operation. Then each unclasp operation on $\Lambda^{(2n)}_{\rm e,o}$ has the effect of removing two homology generators in degree 0. Hence, if we call $\Lambda^{(2n,j)}_{\rm e,o}$ the knot after $j$ unclasps, for $j \in \{0,…,n-1\}$, we get the homology:
$$
LCH^{\bar\varepsilon_1,\bar\varepsilon_2}(\bar\Lambda^{(2n,j)}) =  \Z^{2(n-j)-1}[0] \oplus \Z[-1].
$$
\end{Prop} 

\begin{proof}
We will argue for only one unclasp ($j=1$), the arguments being the same for multiple unclasps. We keep the conventions of the previous paragraphs, and so assume the unclasp operation is done at the level of the double connected sum with the unknots $2i+2$ and $2i+1$, and the removed generators are $b_{2i+1,2i+2}$ and $b_{\rm II,6}$.

The unclasp may change the differential in three ways: first, it removes the image of $b_{2i+1,2i+2}$ and $b_{\rm II,6}$ from the image of the differential. This change only affects the differential in degree 0. Then, it may remove disks which turn in $b_{2i+1,2i+2}$ or $b_{\rm II,6}$. Since both $b_{2i+1,2i+2}$ and $b_{\rm II,6}$ are sent to 0 by both $\varepsilon_1$ and $\varepsilon_2$, any disks which turn in these crossings must contribute exactly them to the differential, and so only the degree 1 differential is affected.
Finally, it may add new disks which pass through the space created by the unclasp, with boundary on the higher strand of what was the unknot $2i+1$ and lower strand of what was the unknot $2i+2$. Such disks cannot exist before the unclasp. One can see in Figure \ref{fig:reidemeister_III} that only the differentials of $b_{\rm II,8}$ and $a_{2i+1}$ may count disks which pass by this space. This is enough to show that only the differentials in degree 0 and 1 are affected by this case. Therefore, the only homology groups which can change are those affected by the degree 0 and degree 1 differential. 

Let us show that the image of the degree 0 differential is unchanged. We will argue for both $\partial^{\varepsilon_1,\varepsilon_2}$ and $\partial^{\varepsilon_2,\varepsilon_1}$, as the other order will be useful later in the proof. Firstly, for both orders of augmentations, the bilinearized differential of $b_{\rm II,6}$ is clearly 0, so removing it does not change the image. As for $b_{2i+1,2i+2}$, we have on one hand $$\partial^{\varepsilon_1,\varepsilon_2} b_{2i+1,2i+2}=b_{2i+1,2i+3}=\partial^{\varepsilon_1,\varepsilon_2}b_{2i+2,2i+3},$$ and on the other hand $$\partial^{\varepsilon_2,\varepsilon_1}b_{2i+1,2i+2}=b_{2i,2i+2}=\partial^{\varepsilon_2,\varepsilon_1}b_{2i,2i+1},$$ (see equation \eqref{eq:computation_diff}), and so in either case respectively $b_{2i+1,2i+3}$ and $b_{2i,2i+2}$ stay in the image after the removal of $b_{2i+1,2i+2}$. Furthermore, as one can see thanks to Figure \ref{fig:disks_unclasp}, the disks which contribute to the differential of $b_{2i+1,2i+2}$ end up in correspondence with the disks which are added after the unclasp and contribute to the differential of $b_{\rm II,8}$. The appearance of such disks then either leaves the bilinearized differentials unchanged or adds $b_{2i+1,2i+3}$ or $b_{2i,2i+2}$ (depending on the order of the augmentations) as a term in the image of the differential of elements to the right of the unclasp, but since $b_{2i+1,2i+3}$ and $b_{2i,2i+2}$ are already in the image, this does not change the vector space $\text{im}( \partial_0^{\varepsilon_1,\varepsilon_2})$. We may then deduce that $LCH^{\varepsilon_1,\varepsilon_2}_{-1}(\Lambda_{\rm e,o}^{(2n,j)})=LCH^{\varepsilon_1,\varepsilon_2}_{-1}(\Lambda_{\rm e,o}^{(2n)})$, and the same holds when we switch augmentations.

 \begin{figure}[h]
 \labellist
\small\hair 2pt
\pinlabel {$b_{2i+3,2i+2}$} [bl] at -15 60 
\pinlabel {$b_{2i+1,2i}$} [bl] at -11 31 
\pinlabel \rotatebox{-20}{$b_{2i+1,2i+3}$} [bl] at 104 50
\pinlabel \rotatebox{30}{$b_{2i,2i+2}$} [bl] at 103 35 
\endlabellist
 \includegraphics[width=8cm, height=7cm]{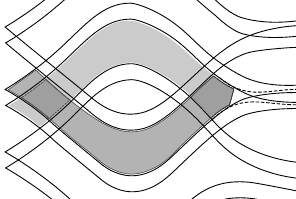}
  \caption{Parts of disks that can contribute to the bilinearized differential and which can appear after the unclasp operation, or in the differential of $b_{2i+1,2i+2}$ before the unclasp.}
  \label{fig:disks_unclasp}
\end{figure}
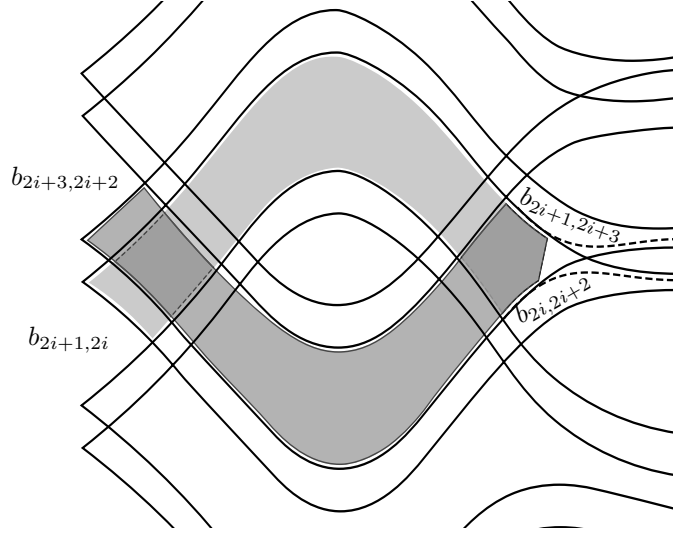

Observe that, by consequence of the long exact sequence \eqref{eq:dualityseq-bilin}, for any Legendrian link $\Lambda$ and pair of augmentations $\varepsilon_1,\varepsilon_2$, $LCH_k^{\varepsilon_1,\varepsilon_2}(\Lambda)\cong LCH^{-k}_{\varepsilon_2,\varepsilon_1}(\Lambda)$ for $k\neq0,-1$. Furthermore, due to the Universal Coefficient Theorem, $$LCH^{k}_{\varepsilon_1,\varepsilon_2}(\Lambda)\cong \Z^{b_k(\Lambda,\varepsilon_1,\varepsilon_2)}\oplus T_{k-1}(\Lambda,\varepsilon_1,\varepsilon_2),$$ where $b_k(\Lambda,\varepsilon_1,\varepsilon_2)$ is the dimension of the free part of $LCH_{k}^{\varepsilon_1,\varepsilon_2}(\Lambda)$ and $T_{k-1}(\Lambda,\varepsilon_1,\varepsilon_2)$ is the torsion part of $LCH_{k-1}^{\varepsilon_1,\varepsilon_2}(\Lambda)$. For $k=1$, this then gives us $LCH_{1}^{\varepsilon_1,\varepsilon_2}(\Lambda)\cong \Z^{b_{-1}(\Lambda,\varepsilon_2,\varepsilon_1)}\oplus T_{-2}(\Lambda,\varepsilon_2,\varepsilon_1)$. To show that the homology in degree 1 is unchanged by the unclasp, it then suffices to show that the homology groups in degree $-1$ and $-2$ are unchanged by the unclasp when we reverse the augmentations. Since the arguments showing that only the differentials in degree 0 and 1 are affected by the unclasp did not depend on the order of the augmentations, the homology group in degree $-2$ is left unchanged regardless of the order of the augmentations we chose. Furthermore, the previous paragraph shows that the homology group in degree $-1$ does not change either. 

The only homology group which can then change due to the unclasp operation is the one in degree 0. Corollary \ref{cor:torsion_iso} shows that for any Legendrian knot $\Lambda$ and augmentations $\varepsilon_1,\varepsilon_2$, $T_0(\Lambda,\varepsilon_1,\varepsilon_2)\cong T_{-1}(\Lambda,\varepsilon_2,\varepsilon_1)$ if $\tau_0$ is surjective. Here we have $\tau_0(b_{2i+1,2i})=\varepsilon_2(b_{2i+1,2i})-\varepsilon_1(b_{2i+1,2i})=1$ for any $i$, which gives us the surjectivity. Since $LCH^{\varepsilon_2,\varepsilon_1}_{-1}(\Lambda_{\rm e,o}^{(2n,j)})=LCH^{\varepsilon_2,\varepsilon_1}_{-1}(\Lambda_{\rm e,o}^{(2n)})$, this then implies that the torsion part of $LCH^{\varepsilon_1,\varepsilon_2}_{0}(\Lambda_{\rm e,o}^{(2n,j)})$ must be the same as the torsion part of $LCH^{\varepsilon_1,\varepsilon_2}_{0}(\Lambda_{\rm e,o}^{(2n)})$, ie must stay trivial. Therefore, only the free part of $LCH^{\varepsilon_1,\varepsilon_2}_{0}(\Lambda_{\rm e,o}^{(2n)})$ can change. Since the Euler characteristic associated to the bilinearized Legendrian contact homology of $\Lambda$ with regards to $\varepsilon_1$ and $\varepsilon_2$ diminishes by two after the unclasp operation, we may then deduce that $LCH^{\varepsilon_1,\varepsilon_2}_{0}(\Lambda_{\rm e,o}^{(2n)})$ must lose two free generators. This concludes our proof.\end{proof}

We perform this unclasp operation $n-1$ times, which per Proposition \ref{prop:unclasp_1} gives us the homology:
$$
LCH^{\bar\varepsilon_1,\bar\varepsilon_2}(\Lambda^{(2n,n-1)}_{\rm e,o}) = \Z[0] \oplus \Z[-1].
$$

Since the unique degree 1 generator in fact corresponds to the alternate sums of the uneven cusps, $\tau_1([a_{\text{odd}}])=[\Lambda_{\text{odd}}]$, the fundamental class of the odd connected component, and so $\rho_1$ is surjective. Therefore, if we perform a usual connected sum between these two knots, we obtain a new Legendrian knot $\Lambda^{\rm min}_{2n}$, for which there exists augmentations which extend the previous ones by sending the new crossing to $-1$, as no disks may go through this crossing without turning in it. So by Proposition~\ref{prop:connsum} we obtain:
$$
LCH^{\bar\varepsilon_1,\bar\varepsilon_2}(\Lambda^{\rm min}_{2n}) = \Z[0].
$$

 \begin{figure}[h]
\labellist
\small\hair 2pt
\pinlabel {$\Lambda_{\rm e,o}^{(6,2)}$} [bl] at 105 210
\pinlabel {$\Lambda_6^{\rm min}$} [bl] at 430 210
\endlabellist
 \includegraphics[width=8cm, height=6cm]{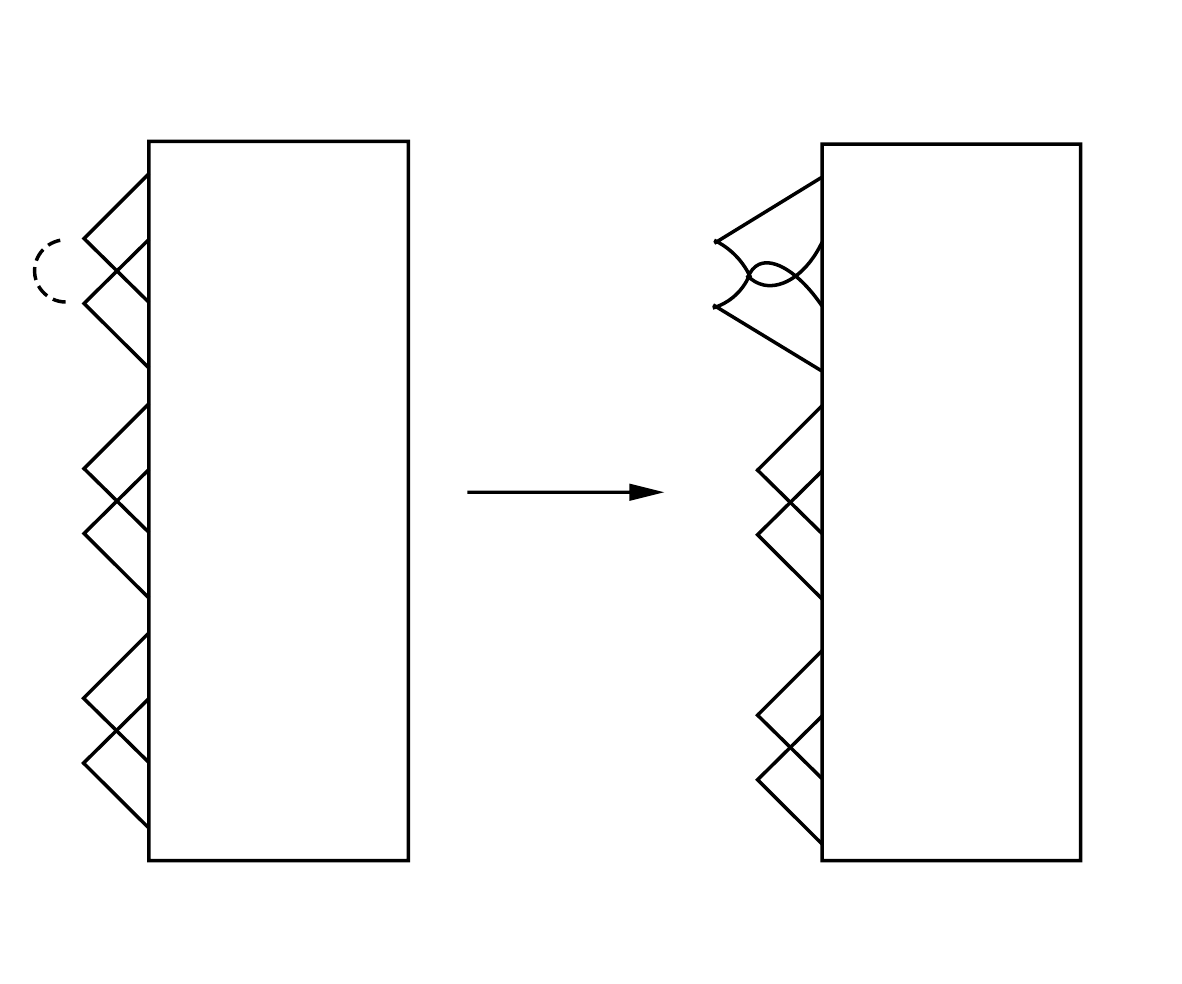}
  \caption{The connected sum taking us from $\Lambda_{\rm e,o}^{(6,2)}$ to $\Lambda_6^{\rm min}$.}
  \label{fig:last_sc}
\end{figure}
\begin{Rem}\label{rmk:looks_like_tr} If we perform the last connected sum at the level of the two top cusps on the left, we obtain something that ressembles a trefoil at the top. See Figure \ref{fig:last_sc}. In fact, for $n=1$, this recovers the construction of the trefoil knot.\end{Rem}
 \subsection{Building the free part of bilinearized Legendrian contact homology}
 \label{sec:free_geog_bilin}

We now follow the constructions introduced in~\cite{BG} in order to create pairs of free generators in arbitrary degrees. This construction was described
in~\cite[Section~4.4]{BG} with coefficients in $\Z_2$; we simply have to verify that this construction does not generate any torsion when performed in 
dimension $3$ with coefficients in $\Z$. To this end, for any $k = 1, \ldots, 2n$, it suffices to consider a single standard Legendrian unknot 
$\Lambda_{0}$ interlaced by the $k$ bottom strands of $\Lambda^{\rm min}_{2n}$, corresponding to the bottom strand of each of the former components $\Lambda_1,
\ldots, \Lambda_k$. We obtain $2k+1$ additional generators in the bilinearized complex: the right cusp $a'_0$ of $\Lambda_0$, the intersection of
the upper strand of $\Lambda_0$ with the $i$th bottom strand of $\Lambda^{\rm min}_{2n}$ as the former is going up: $c'_i$, and as it
is going down: $d'_i$, for $i = 1, \ldots, k$. This is illustrated by Figure~\ref{fig:free_geog_bilin}.

\begin{figure}
\labellist
\small\hair 2pt
\pinlabel {$\Lambda_0$} [bl] at 20 0
\pinlabel {$c'_1$} [bl] at 70 100
\pinlabel {$d'_1$} [bl] at 215 100
\pinlabel {$c'_2$} [bl] at 75 130
\pinlabel {$d'_2$} [bl] at 210 130
\pinlabel {$c'_k$} [bl] at 90 185
\pinlabel {$d'_k$} [bl] at 195 185
\pinlabel {$a_0$} [bl] at 290 20
\pinlabel {$\left.  \rule{0pt}{0.9cm} \right\} k$}  [bl] at 310 93
\endlabellist
  \centerline{\includegraphics[width=5cm]{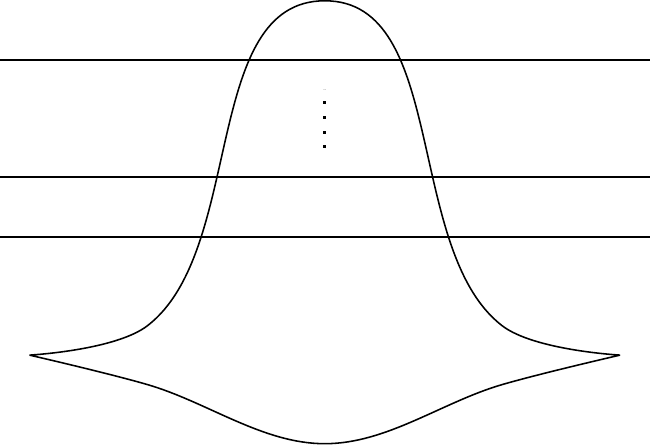}}
  \caption{A standard Legendrian unknot interlaced with $k$ parallel strands.}
  \label{fig:free_geog_bilin}
\end{figure}
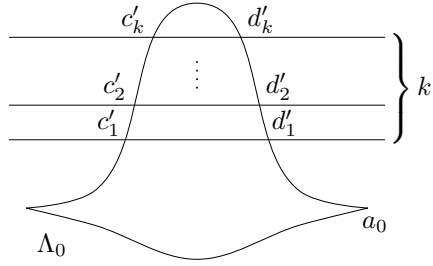

The gradings of these generators are $|a_0|= 1$, $|c'_i| = |c'_1|+i-1$ and $|d'_i| = -|c'_i|$ for $i= 1, \ldots, k$. We extend the augmentations 
$\bar\varepsilon_1 $ and $\bar\varepsilon_2$ of $\Lambda^F_{2n}$ by zero on these new generators. For all $i = 1, \ldots, k$, 
there is a unique rigid holomorphic disk with a left-facing corner at $c'_i$ or $d'_i$ and with at most one unaugmented upward-facing corner:
it has upward-facing corners at $c'_{i-1}$ or $d'_{i+1}$ and at $b_{i,i-1}$ or $b_{i+1,i}$.
The bilinearized differential of the new generators is therefore given by 
$$
\partial^{\bar\varepsilon_1,\bar\varepsilon_2} c'_i = -\bar\varepsilon_1(b_{i,i-1}) c'_{i-1}
\qquad \textrm {and} \qquad 
\partial^{\bar\varepsilon_1,\bar\varepsilon_2} d'_i = (-1)^{|d'_{i+1}|}\bar\varepsilon_2(b_{i+1,i}) d'_{i+1},
$$
for $i = 1, \ldots, k$, and $\partial^{\bar\varepsilon_1,\bar\varepsilon_2} a_0 = 0$. 
Hence, the bilinearized homology has three additional free generators: $[a_0]$ in degree $1$, $[c'_{k}]$ if $k$ is odd 
or $[d'_{k}]$ if $k$ is even, and $[d'_{1}]$. By varying $k$ and the Maslov potential of $\Lambda_0$, we can realize any pairs of degrees
for the last two generators. 

Finally, we perform a connected sum between $\Lambda^{\rm min}_{2n}$ and $\Lambda_0$. Due to the discussion before Proposition~\ref{prop:connsum}, the augmentations extend to the connected sum by sending the new crossing to $-1$, as no disk can appear which goes through the new crossing without turning in it. Since $\tau_1([a_0]) = [\Lambda_0]$, the map $\rho_1$ is surjective
and by Proposition~\ref{prop:connsum} the effect of this connected sum is to remove the free generator $[a_0]$ of degree $1$ from bilinearized LCH.
We therefore obtained the desired effect of this operation on bilinearized LCH, so that as in~\cite{BG} its free part can be any graded free module of odd rank with at least 
one generator of degree $0$.

 \subsection{Building the torsion part of bilinearized Legendrian contact homology.}\label{subsection:building_torsion}

 We consider the following link $\Lambda_{m,k}$, as seen in Figure \ref{fig:figure_torsion}: a trefoil knot interlaced with an unknot, such that the front projections intersect at the right ``eye'' of the trefoil knot. The trefoil knot has a Maslov potential shift by $k$ with regards to the unknot, so that $|b_7|=-|b_6|=k$.
\begin{figure}
\labellist
\small\hair 2pt
\pinlabel {$b_1$} [bl] at 67 54
\pinlabel {$b_2$} [bl] at 137 46
\pinlabel {$b_3$} [bl] at 210 54
\pinlabel {$b_5$} [bl] at 192 81
\pinlabel {$b_6$} [bl] at 110 118
\pinlabel {$b_7$} [bl] at 212 121
\pinlabel {$b_4$} [bl] at 148 79
\pinlabel {$a_1$} [bl] at 267 142
\pinlabel {$a_2$} [bl] at 267 100
\pinlabel {$a_3$} [bl] at 267 42
\endlabellist
  \centerline{\includegraphics[width=7cm]{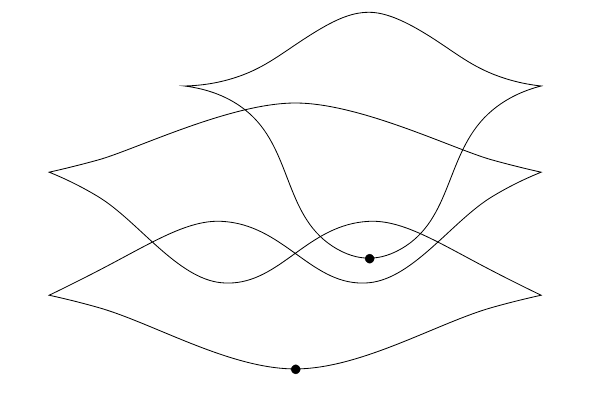}}
  \caption{A trefoil knot interlaced with an unknot, each with a base point.}
  \label{fig:figure_torsion}
\end{figure}
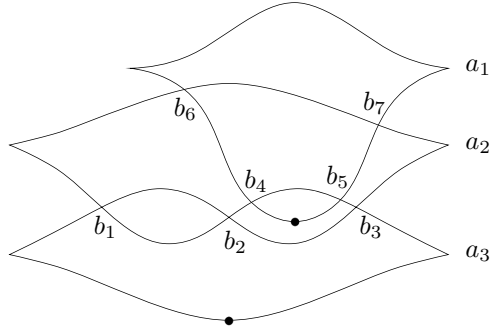

We choose two augmentations $\varepsilon_1$ and $\varepsilon_2$ which send every crossing to $0$, except for:
\begin{align*}
\varepsilon_1(b_1)&=-1 & \varepsilon_2(b_1)&=0,\\
\varepsilon_1(b_2)&=m\in\Z\setminus\{0,1,2\} & \varepsilon_2(b_2)&=0,\\
\varepsilon_1(b_3)&=0 & \varepsilon_2(b_3)&=-1.
\end{align*}
Also, both augmentations send each $t_i,i=1,2,3$, to $-1$.
With this choice of augmentations, we have the following differentials:
\begin{align*}
\partial^{\varepsilon_1,\varepsilon_2} a_1 &= 0,\\
\partial^{\varepsilon_1,\varepsilon_2} a_2 &= b_1 + b_3 + b_2 - mb_3,\\
\partial^{\varepsilon_1,\varepsilon_2} a_3 &= b_1 + b_3,\\
\partial^{\varepsilon_1,\varepsilon_2} b_7 &= (m-1)b_5,\\
\partial^{\varepsilon_1,\varepsilon_2} b_4 &= (-1)^{k}b_6,\\
\partial^{\varepsilon_1,\varepsilon_2} b_1 &= \partial^{\varepsilon_1,\varepsilon_2} b_2 = \partial^{\varepsilon_1,\varepsilon_2}b_3=0 = \partial^{\varepsilon_1,\varepsilon_2}b_5 =\partial^{\varepsilon_1,\varepsilon_2}b_6=0.
\end{align*}
We therefore have a unique free degree 1 generator, $[a_1]$, a unique free degree 0 generator, $[b_1]=-[b_3]$ (and $[b_2]=m[b_3]$), and a unique torsion generator, $b_5$, which gives us $(m-1)$-torsion in degree $k-1$:

$$
LCH^{\varepsilon_1, \varepsilon_2}(\Lambda_{m,k}) = \Z[0] \oplus \Z[-1] \oplus \Z/(m-1)\Z[1-k].
$$

In order to remove the unnecessary free generators, we proceed similarly to the case of the $2n$-copy. We interlace $\Lambda_{m,k}$ with another Legendrian trefoil knot $\Lambda_{\rm tr}$, such the front projection of $\Lambda_{m,k}$ intersects the left eye of $\Lambda_{\rm tr}$, as seen in Figure \ref{fig:figure_torsion_link}, to get a link $\tilde\Lambda^T_{m,k}$.

\begin{figure}[h]
  \centerline{\includegraphics[width=7cm]{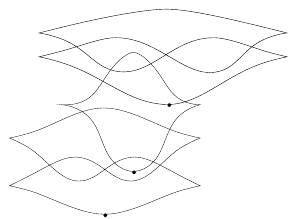}}
  \caption{Link constituted of $\Lambda_{\rm tr}$ and $\Lambda_{m,k}$.}
  \label{fig:figure_torsion_link}
\end{figure}

We extend the augmentations to the link by sending the far-left crossing of $\Lambda_{\rm tr}$ to $-1$ by each augmentation, so the crossings between $\Lambda_{\rm tr}$ and $\Lambda_{m,k}$ cancel out in homology. We also send the middle crossing of $\Lambda_{\rm tr}$ to $1$ for both augmentations. $\Lambda_{\rm tr}$ therefore contributes one generator in degree 1, and two generators in degree 0, so that we have the following homology for the link:

\begin{align*}
LCH^{\varepsilon_1, \varepsilon_2}(\tilde\Lambda^T_{m,k}) &= LCH^{\varepsilon_1, \varepsilon_2}(\Lambda_{m,k}) \oplus \Z^2[0] \oplus \Z[-1] \\
&= \Z^3[0] \oplus \Z^2[-1] \oplus \Z/(m-1)\Z[1-k].
\end{align*}

We now perform a double connected sum, one between the unknot in $\Lambda_{m,k}$ and $\Lambda_{\rm tr}$, at the level of the top right cusp of $\Lambda_{\rm tr}$, and one between $\Lambda_{m,k}$, at the level of its top right cusp, and $\Lambda_{\rm tr}$, at the level its the bottom right cusp, which gives us a knot $\Lambda^T_{m,k}$. More precisely, we start by performing a series of Reidemeister I moves on the lower knot, near the top cusp, so that the top strand has the Maslov potential of the lower strand of the unknot. We then perform a two-copy of a Reidemeister move on the lower two knots, followed by two Reidemeister two moves so that the upper strand of the unknot crosses twice the second lower strand of $\Lambda_{\rm tr}$. We can now do a double connected sum, one between the unknot in $\Lambda_{m,k}$ and $\Lambda_{\rm tr}$, and one between the trefoil knot in $\Lambda_{m,k}$ and $\Lambda_{\rm tr}$. See Figure \ref{fig:figure_torsion_unclasp_2}. The augmentations extend after the connected sum by sending the new crossings each to -1.

Per Proposition \ref{prop:connsum}, the connected sum between the unknot and the $\Lambda_{\rm tr}$ removes a free degree 1 generator, since the image by $\tau_{1,1}-\tau_{1,2}$ of the fundamental class of the unknot is 1. Once this connected sum has been performed, we have a link formed of two knots, one of whom (obtained by the first connected component) carries a fundamental class, which is obtained from the sum of the fundamental class of the unknot and the fundamental class of $\Lambda_{\rm tr}$, and the other which does not. Again, the image of the fundamental class by $\tau_{1,1}-\tau_{1,2}$ is 1, and so the morphism is surjective. The second connected sum then also removes a degree 1 generator. 
We are left with a $(m-1)$-torsion generator in degree $k$ and three free generators in degree 0. To obtain the geography, we want a single degree 0 generator, so we must get rid of two. We once again perform the unclasp operation like in Subsection \ref{subsection:2n-copy}. That is, we get rid of two crossings, the very right one of $\Lambda_{\rm tr}$ ($b_8$) and the one on the right of it created by the double connected sum ($e_4$), by raising and/or lowering one of the two strands to which these crossings belong to. See Figure \ref{fig:figure_torsion_unclasp_2}.

\begin{figure}
\labellist
\small\hair 2pt
\pinlabel {$b_{1}$} [bl] at 30 25
\pinlabel {$b_{2}$} [bl] at 64 20
\pinlabel {$b_{3}$} [bl] at 102 28
\pinlabel {$b_{4}$} [bl] at 70 42
\pinlabel {$b_{5}$} [bl] at 96 52
\pinlabel {$b_{6}$} [bl] at 26 60
\pinlabel {$b_{7}$} [bl] at 194 51
\pinlabel {$b_{8}$} [bl] at 136 104
\pinlabel {$b_{9}$} [bl] at 84 110
\pinlabel {$b_{10}$} [bl] at 28 114
\pinlabel {$a_{1}$} [bl] at 128 20
\pinlabel {$a_{2}$} [bl] at 214 49
\pinlabel {$a_{3}$} [bl] at 214 59
\pinlabel {$a_{k+2}$} [bl] at 206 102
\pinlabel {$a_{k+3}$} [bl] at 206 110
\pinlabel {$a_{k+4}$} [bl] at 206 116
\pinlabel {$b_{\rm II,1}$} [bl] at 195 93
\pinlabel {$b_{\rm II,8}$} [bl] at 183 114
\pinlabel {$b_{\rm II,3}$} [bl] at 156 94
\pinlabel {$b_{\rm II,6}$} [bl] at 168 108
\endlabellist
  \centerline{\includegraphics[width=11cm]{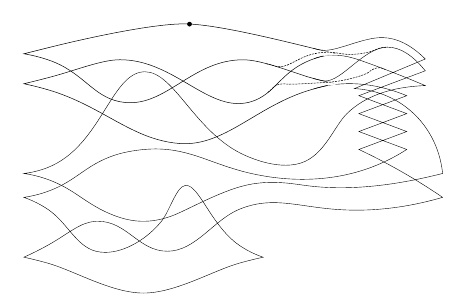}}
  \caption{The unclasp operation: the dotted lines replace the whole ones.}
  \label{fig:figure_torsion_unclasp_2}
\end{figure}

Let us check that the unclasp has the desired effect, ie. deleting two degree 0 generators. The computations for the augmentations are similar to subsection \ref{subsection:unclasp}. Indeed, the series of Reidemeister I moves on the lower knot creates a series of crossings which are all sent to $-1$ by the augmentations, as well as the cusp created by the Reidemeister I move on the unknot. We then perform two Reidemeister II moves, which create four new cusps $b_{\rm II,1}$ to $b_{\rm II,4}$. See Figure \ref{fig:double_reidemeister_1}. Note that this time, as opposed to subsection \ref{subsection:unclasp}, only one disk contributes to the differential of $b_{\rm II,1}$, with a corner in $b_{\rm II,2}$. Same for the differential of $b_{\rm II,4}$, which counts only a disk with a corner in $b_{\rm II,3}$. The DGA map given in \cite[section 6.3.3.]{EKK} then sends $b_{\rm II,1},…,b_{\rm II,4}$ to 0, and so the augmentation given by the pullback of the previous augmentations by this map send them all to 0. The following Reidemeister III move then preserves all the augmentations. The next moves are as in Figure \ref{fig:pre_unclasp}. We perform again two Reidemeister II moves. Like in subsection \ref{subsection:unclasp}, the augmentations send all these new crossings to 0. We then do two connected sums, which create two new crossings which are each sent to $-1$ by either augmentation, and a Reidemeister III move which does not change the values of the augmentations. We therefore end up with a similar situation to subsection \ref{subsection:unclasp}, where in particular, both cusps removed by the unclasp are sent to 0 by either augmentation. 

We may then deduce that Lemma \ref{lem:augmentations_unclasp} still holds in this setup. Indeed, the arguments used in the proof depend only on what the knot looks like to the right of the unclasp, which is the same in this case, and on the fact that the two generators which are removed by the unclasp get sent to 0 by the augmentations. Therefore, the unclasp does not change the values of the augmentations of the generators which are not removed.

We can now prove that the unclasp has the desired effect of removing two degree 0 generators. 

\begin{Prop}\label{prop:unclasp_2}
Let $\bar\varepsilon_1$, $\bar\varepsilon_2$ be the augmentations after the unclasp operation. Then the unclasp operation has for only effect the removal of two free degree 0 generators from $LCH^{\bar\varepsilon_1, \bar\varepsilon_2}_{0}(\Lambda^T_{m,k})$ . \end{Prop}

\begin{proof}
The proof is the same as for Proposition \ref{prop:unclasp_1}, every argument adapting directly, with two minor differences: in the previous case, the generator which appeared in the image of the left crossing removed by the unclasp already appeared in the image of $\partial_0^{\bar\varepsilon_1,\bar\varepsilon_2}$, and so the image did not change after the unclasp. Here, the argument is in fact easier, as the image of $b_8$ by the differential is 0, because there are no disks with positive punctures at $b_8$, so its removal automatically does not change the image of $\partial_0^{\bar\varepsilon_1,\bar\varepsilon_2}$. This also shows that no new disks may appear after the unclasp, as the new disks which can appear are in correspondence with the disks which have a positive puncture in $b_8$. The other argument to adapt is the argument of the surjectivity of $\tau_0$: here, we have $\tau_0(b_1)=\varepsilon_2(b_1)-\varepsilon_1(b_1)=1$, which guarantees that the morphism is surjective. Every other argument directly adapts to this new case, and so we may conclude that the unclasp operation has for only effect on the homology the removal of two free degree 0 generators.
\end{proof}

By Proposition \ref{prop:unclasp_2} and the discussion which precedes it, performing the double connected sum gets rid of two degree 1 generators, and the unclasp removes two degree 0 generators. We therefore obtain the following homology: 
$$
LCH^{\varepsilon_1, \varepsilon_2}(\Lambda^T_{m,k}) = \Z[0] \oplus \Z/(m-1)\Z[1-k].
$$

We can repeat this operation by linking another $\Lambda_{m',k'}$ into what was the right eye of $\Lambda_{\rm tr}$. We then do the same double connected sum at the level of the two top cusps of $\Lambda^T_{m,k}$. This time, only the connected sum with the unknot of $\Lambda_{m',k'}$ leads to the deletion of a degree 1 generator, while the other leads to the addition of a degree 0 generator. This therefore gives us a new knot which has one $(m-1)$-torsion generator in degree $k-1$, and three free degree 0 generators. 

We do need to make sure the unclasp has the same effect as in Proposition \ref{prop:unclasp_2} even when repeating this operation. One can see that the computations of the augmentations and the proof of Proposition \ref{prop:unclasp_2} still hold, as they only depend on what happens to the right of the unclasp, and on the differential of $b_{\rm II,8}$, which takes on the role of $b_{8}$ in the following unclasp.

In fact, this is true every time we add a new $\Lambda_{m',k'}$ with $m' \neq 0, 1, 2$ to the construction, which allows us to add $(m-1)$-torsion in any degree for any $m\in\Z$, without modifying the the dimension of the other homology groups. We therefore can perform an unclasp every time which removes two degree 0 generators, and obtain a knot $\Lambda^T_{(\overline{m},\overline{k})}$, with $\overline{m} = (m_1, \ldots, m_N)$ and $\overline{k} = (k_1, \ldots, k_N)$,  which has the homology 
$$
LCH^{\varepsilon_1, \varepsilon_2}(\Lambda^T_{(\overline{m},\overline{k})}) = \Z[0] \oplus \bigoplus_{i=1}^N \Z/(m_i-1)\Z[1-k_i].
$$
See Figure \ref{fig:torsion_full}.

\begin{figure}
\labellist
\small\hair 2pt
\pinlabel {$m_{1}$} [bl] at 24 9
\pinlabel {$m_{2}$} [bl] at 112 27
\pinlabel {$m_{3}$} [bl] at 178 32
\endlabellist
  \centerline{\includegraphics[width=10cm,height=5cm]{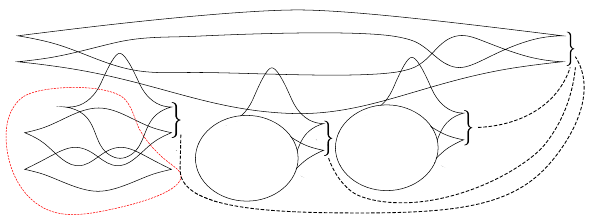}}
  \caption{$\Lambda^T_{(\overline{m},\overline{k})}$ for $\bar{m}=(m_1,m_2,m_3)$. The dotted lines represent the double connected sum and unclasp operations.}
  \label{fig:torsion_full}
\end{figure}

Notice that the case where $m=2$ or $0$ contributes no new generators, and the case where $m=1$ contributes two free generators, one in degree $k$ and one in degree $k-1$.

 \subsection{Proof of Theorem~\ref{thm:bilinLCH}}

To complete the geography, we want to find a way to add free generators, in pairs (to respect the Euler characteristic of the bilinearized Legendrian homology of the knot), but not necessarily in consecutive degrees. In particular, we want to attach the $2n$-copy $\Lambda^{\rm min}_{2n}$ from the construction in Subsection \ref{subsection:2n-copy}. In this case, we choose the final connected sum in the construction of $\Lambda^{\rm min}_{2n}$ to be at the level of the two top left cusps. 

To attach $\Lambda^{\rm min}_{2n}$ to $\Lambda^T_{(\overline{m},\overline{k})}$, note that performing a double connected sum and unclasp with only these two knots would not be effective: the double connected sum would have the effect of adding two degree 0 generators (as there are no degree 1 generators to remove), and the unclasp having an effect of removing two from the Euler characteristic, we cannot hope for the desired outcome (that is, ending up with a single free degree 0 generator and then adding free generators thanks to $\Lambda^{\rm min}_{2n}$). We may however use a slightly different link to then use this method.

First, recall (Remark \ref{rmk:looks_like_tr}) that the top of of $\Lambda^{\rm min}_{2n}$ looks like a trefoil. Notice as well that the proof of Proposition \ref{prop:unclasp_2} does not ``see'' what happens below the level of the unclasp. Therefore, we can replace the trefoil of $\Lambda_{m,k}$ with $\Lambda^{\rm min}_{2n}$, and have the same effect when we attach it to $\Lambda_{\rm tr}$. More precisely, we look at the link formed of $\Lambda^{\rm min}_{2n}$ with an unknot linked with the top two strands, linked with the right eye. Adding the unknot to the link adds a free degree 1 generator to the homology. We then interlink the unknot with the left eye of $\Lambda_{\rm tr}$, as in the construction of $\Lambda^T_{(\overline{m},\overline{k})}$. Then, we do a double connected sum, at the level of the two cusps of $\Lambda_{\rm tr}$ and the cusp of the unknot and top cusp of $\Lambda^{\rm min}_{2n}$. As in the case of $\Lambda^T_{(\overline{m},\overline{k})}$, the connected sum involving the unknot removes a free degree 1 generator, whereas the other connected sum adds a free degree 0 generator. Then we do the unclasp. Since this case is analogous to the one proved in Proposition \ref{prop:unclasp_2}, this has the effect of removing two degree 0 generators. 

We have therefore attached the $2n$-copy to $\Lambda_{m,k}$ in a way that still allows us to have the minimal amount of free generators (a unique degree 0 generator), but allows us to then attach free generators in pairs of arbitrary degrees, like we did in Subsection~\ref{sec:free_geog_bilin}. See Figure \ref{fig:torsion_free_full}.

For any graded $\Z$-module $M$ with at least one free generator in degree 0 and an odd number of free generators, we can therefore construct a knot $\Lambda_M$ and two augmentations $\varepsilon_1$ and $\varepsilon_2$ constructed in such a way, such that $LCH^{\varepsilon_1, \varepsilon_2}_{\ast}(\Lambda_{M}) \cong M$. Conversely, any Legendrian knot whose DGA admits two non-dga-homotopic augmentations must be 
of this form. This proves Theorem~\ref{thm:bilinLCH}. 

\begin{figure}
\labellist
\small\hair 2pt
\pinlabel {$m$} [bl] at 180 54
\endlabellist
  \centerline{\includegraphics[width=10cm, height=5cm]{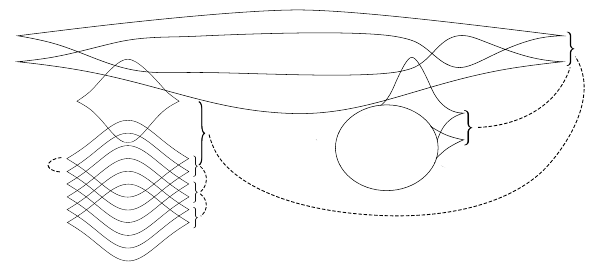}}
  \caption{Attaching $\Lambda^{\rm min}_{6}$ to $\Lambda^T_{(\overline{m},\overline{k})}$. The dotted lines represent connected sums, or double connected sums and unclasps when paired with brackets.}
  \label{fig:torsion_free_full}
\end{figure}

 \end{document}